\RequirePackage{fix-cm}

\documentclass[smallextended]{svjour3}
\smartqed
\usepackage{graphicx}

\usepackage{booktabs}
\usepackage{amsfonts}
\usepackage{geometry}

\usepackage{amsfonts}
\usepackage{amsmath}
\usepackage{csquotes}
\usepackage{algorithmicx}
\usepackage{algorithm}
\usepackage{algpseudocode}
\usepackage{color}

\usepackage{empheq}
\usepackage{physics}
\usepackage{comment}

\newcommand{\RR}{\mathbb{R}}

\begin{document}

\title{Weighted Derivative Histopolation on Arbitrary Grids:
Admissibility and Exact Factorizations
}

\titlerunning{Weighted Derivative Histopolation on Arbitrary Grids}        

\author{Allal Guessab \and
        Federico Nudo  
}


\institute{Allal Guessab \at
       Laboratoire de Mathématiques et de leurs Applications, UMR CNRS 5142, Université de Pau et des Pays de l'Adour (UPPA), France\\ 
             \email{allal.guessab@univ-pau.fr} 
         \and
            Federico Nudo (corresponding author) \at
              Department of Mathematics and Computer Science, University of Calabria, Rende (CS), Italy\\
              \email{federico.nudo@unical.it}
}

\date{Received: date / Accepted: date}

\maketitle

\begin{abstract}
In this paper, we introduce a weighted derivative histopolation framework on
families of intervals. The degrees of freedom consist of one scalar
normalization and weighted integral moments of the derivative over a
prescribed family of subintervals. We prove that the resulting scheme is
unisolvent on $\Pi_N$ when the interval family separates polynomials
of degree at most $N-1$ through weighted moments and the normalization is
nonzero on constants. Thus, the derivative moments determine the polynomial up to an additive constant, and the scalar normalization fixes this remaining degree of freedom. This gives a sharp criterion for the
well-posedness of the interpolation problem and a complete characterization
of the admissible scalar normalizations. We then show how admissible
families of intervals can be constructed from a fixed grid. When the
endpoints of the intervals belong to the grid, admissibility is reduced to
the nonsingularity of an interval matrix associated with the family, which
depends only on the representation of the intervals in terms of consecutive
cells. For Jacobi weights, the associated data matrices have a natural block
structure in Jacobi polynomial bases, and the reduced derivative matrix can
be expressed in terms of shifted Jacobi moment matrices. We next study
Chebyshev configurations in which this structure becomes explicit. For the four classical Chebyshev families, suitable polynomial bases lead to
diagonal Gram matrices for the reduced derivative matrices. We show that this
diagonal structure depends on the simultaneous choice of the weight, the
basis, and the grid. In particular, for general Jacobi weights on a fixed
Chebyshev--Lobatto grid, the corresponding discrete orthogonality is not
available in general, although an identity Gram matrix can always be obtained
by a change of basis. We also identify configurations in which the reduced
derivative matrices admit exact sine transform factorizations, yielding
explicit singular values and spectral condition numbers. Numerical experiments on equispaced and
Chebyshev--Lobatto nodes show the behaviour of the method for different
interval families and for different Jacobi parameters.
\end{abstract}

\section{Introduction}
\label{sec:introduction}

Approximation from integral or averaged data is a natural alternative to
classical interpolation when pointwise values are not available, or when they
do not represent the information provided by the measurement
process~\cite{Bruno:2024:PIO,Guessab:2026:SDO}. This situation occurs in
several applications in which the data are integrals over intervals, cells,
faces, or more general geometric regions, as in tomography and image
reconstruction~\cite{Kak:2001:POC,Palamodov:2016:RFI}. In these cases, the
reconstruction procedure should be consistent with the integral nature of the
data. This is the basic principle of histopolation, where the degrees of
freedom are given by averages or moment type functionals rather than by point
evaluations~\cite{Schoenberg}. In this sense, histopolation is a particular
case of polynomial interpolation defined by linear functionals. In the
classical Lagrange case, these functionals are evaluations at prescribed
nodes; in the histopolation case, they are integral averages over prescribed
subintervals~\cite{Davis:1975:IAA}. Several polynomial histopolation schemes have been introduced and studied in
recent years~\cite{Bruni:2025:TFP,Bruni:2025:PHO,Bruno:2026:BPH}. When the
integral functionals are taken with respect to a weight, this leads to
weighted histopolation~\cite{Nudo:2026:FRU,DellAccio:2026:NAM}. Such
constructions arise when the averaged data are taken with respect to a
nonuniform density, represented by the weight function. In geometric design,
for instance, averaging operators are often used to preserve qualitative
properties such as positivity and monotonicity~\cite{demichelis1995graphic}.
Related average data reconstruction problems also occur in sampling and
reconstruction on combinatorial graphs~\cite{furh}. In contrast to the
unweighted case, the presence of a weight changes the moment conditions and
may affect fundamental structural properties, including
unisolvence~\cite{Guessab:2026:QWH}. Establishing unisolvent weighted
frameworks is therefore a nontrivial problem, both from a theoretical and a
computational viewpoint.

In this paper, we focus on the univariate case and consider data given by
weighted moments of the derivative over a family of intervals, rather than by
weighted moments of the function. Thus, given $f\in C^1([a,b])$, the degrees
of freedom have the form
\[
\int_s f'(x)\omega(x)\dd x,
\quad s\subset [a,b].
\]
The resulting functionals are still integral functionals, but they act on the
derivative and therefore measure weighted averaged variations of the
function. Derivative information is
used in several approximation and reconstruction procedures in which first
order data are available together with, or instead of, pointwise function
values. This occurs, for example, in methods based on Hermite type
data~\cite{Johnson:1975:TTH,Finden:2007:HOA,Finden:2008:AET} and in
gradient-augmented polynomial
approximation~\cite{Nave:2010:AGA,Adcock:2019:CHI}. Integral data involving
derivatives also appear in classical interpolation theory. In particular,
Kergin interpolation may be regarded as a multivariate interpolation process
in which unweighted integral expressions involving derivatives over simplices
play a fundamental role~\cite{Kergin:1980:APC}. This type of interpolation has
been developed and used in several contexts, ranging from constructive
representations through multivariate B-splines to structural and complex
analytic aspects of interpolation on the disk; see~\cite{micchelli1980constructive,bos1983kergin,bloom1997kergin,petersson2002kergin,Barrera:2017:HSI,Barrera:2021:ONM}. The present paper takes this point of view in a different direction. The
averaged differential data are weighted and are prescribed on a family of
intervals. Since derivative moments do not determine additive constants, the
first question is to understand when they determine a polynomial up to this
remaining degree of freedom, and which scalar normalization fixes it. The second question is to identify configurations for which the associated
matrices admit explicit algebraic decompositions, so that stability and
conditioning can be analyzed directly.

Without loss of generality, we work on the interval $[-1,1]$. Let
\[
\mathcal S=\left\{s_1,\ldots,s_N\right\},
\quad
s_i\subset[-1,1],
\]
be a family of nondegenerate intervals, and let
$\omega\in L^1(-1,1)$ be positive almost everywhere. We consider the
functionals
\[
\mathcal D_i^\omega(f)
=
\int_{s_i} f'(x)\omega(x)\dd x,
\quad i=1,\ldots,N.
\]
Since the derivative of a constant polynomial vanishes, the functionals
$\mathcal D_i^\omega$ cannot determine the constant component of a polynomial.
Therefore, a scalar normalization has to be added.

The first part of the paper is devoted to the unisolvence of this
construction. We introduce the notion of an $\omega$-admissible family of
intervals. This means that the weighted moments over the intervals in
$\mathcal S$ uniquely determine every polynomial in $\Pi_{N-1}$. We prove that
this condition characterizes the families for which the derivative moments
determine a polynomial in $\Pi_N$ up to an additive constant. Consequently,
once the interval family $\mathcal S$ is $\omega$-admissible, if a linear
functional $\mathcal M$ is added as normalization, the degrees of freedom
\[
\left\{
\mathcal M,
\mathcal D_1^\omega,\ldots,\mathcal D_N^\omega
\right\}
\]
are unisolvent on $\Pi_N$ if and only if $\mathcal M(1)\neq0$. Equivalently, the derivative moments determine the polynomial up to an
additive constant, and the scalar normalization fixes this remaining degree of
freedom. This gives a complete characterization of the scalar normalizations
which ensure uniqueness. We then study how admissible families of intervals can be
generated from a fixed grid. For families whose endpoints belong to the grid,
admissibility is reduced to a finite dimensional matrix condition depending
only on how the intervals are formed from the grid. We next consider Jacobi weights. This choice is natural because it is
associated with Jacobi polynomial bases and includes several classical
weights. In this setting, the data matrix has a block form which separates the
constant polynomial from the nonconstant part. The reduced derivative matrix
is the relevant matrix for the nonconstant component, and it can be expressed
in terms of shifted Jacobi moment matrices. This gives an algebraic
description of the weighted problem in bases adapted to the weight. A more explicit description is obtained in Chebyshev configurations. For the
four classical Chebyshev families, suitable polynomial bases lead to diagonal
Gram matrices for the reduced derivative matrices. We emphasize that these
diagonal identities depend on the simultaneous choice of the weight, the
basis, and the grid. They are not a direct consequence of the continuous
orthogonality of Jacobi polynomials on the whole interval. Rather, they come
from the discrete trigonometric orthogonality associated with the
Chebyshev--Lobatto cells. For a fixed pair $(\alpha,\beta)$, the Jacobi--Gauss--Lobatto grid associated
with the weight $\omega_{\alpha,\beta}$ can be used instead, with interior
nodes given by the zeros of the Jacobi polynomial with shifted parameters
$(\alpha+1,\beta+1)$. In that case, the interval family depends on
$(\alpha,\beta)$. We also identify configurations in
which the reduced derivative matrices admit exact sine transform
factorizations. These factorizations yield explicit singular values and
spectral condition numbers.

The paper is organized as follows. In Section~\ref{sec2} we introduce the
weighted derivative moment functionals, define admissible families of
intervals, and prove the unisolvence result on the polynomial space $\Pi_N$.
We also give a matrix characterization of admissibility for interval families
generated from a fixed grid, and use it to construct several admissible
configurations. In Section~\ref{sec3} we analyze the construction for Jacobi
weights. We derive the block form of the data matrix in the Jacobi basis and
express the reduced derivative matrix in terms of shifted Jacobi moment
matrices. We also study Chebyshev configurations in which the reduced Gram
matrices are diagonal in suitable polynomial bases, and identify cases where
exact sine transform factorizations give explicit singular values and
spectral condition numbers. Section~\ref{sec4} uses weighted Jacobi
primitives to relate derivative histopolation to nodal interpolation on
Chebyshev grids. Section~\ref{sec5} contains the numerical experiments, which
illustrate the behaviour of the method for different interval families, grids,
and Jacobi parameters.

\section{Weighted derivative histopolation and unisolvence}
\label{sec2}
Let $N\geq1$, and let
\[
\mathcal S=\left\{s_1,\ldots,s_N\right\},
\quad
s_i=\left[a_i,b_i\right]\subset[-1,1],
\quad a_i<b_i,
\]
be a family of nondegenerate intervals. Let
$\omega\in L^1(-1,1)$ satisfy
\begin{equation}\label{eq:positive-weight}
\omega(x)>0
\quad\text{for almost every }x\in(-1,1).
\end{equation}
For each $i=1,\ldots,N$, we define the weighted derivative moment functional
\begin{equation}\label{eq:derivative-functional}
\mathcal D_i^\omega:
f\in C^1([-1,1])
\to
\int_{s_i}f'(x)\omega(x)\dd x
\in\RR,
\end{equation}
and the corresponding weighted moment functional
\begin{equation*}
\mathcal H_i^\omega:
g\in C([-1,1])
\to
\int_{s_i}g(x)\omega(x)\dd x
\in\RR.
\end{equation*}
We introduce the weighted derivative moment map
\begin{equation}\label{eq:admissible-derivative-map}
\mathcal D_{\mathcal S}^{\omega}:
f\in C^1([-1,1]) 
\to
\left[
\mathcal D_1^\omega(f),
\ldots,
\mathcal D_N^\omega(f)
\right]^\top
\in\RR^N, 
\end{equation}
and the weighted moment map
\begin{equation}\label{eq:admissible-moment-map}
\mathcal H_{\mathcal S}^{\omega}:
g\in C([-1,1])
\to
\left[
\mathcal H_1^\omega(g),
\ldots,
\mathcal H_N^\omega(g)
\right]^\top
\in\RR^N.
\end{equation}
The family $\mathcal S$ is said to be $\omega$-admissible if the restriction
\[
\mathcal H^{\omega}_{{\mathcal S}_{|{\Pi_{N-1}}}}
:
\Pi_{N-1}\to\RR^N
\]
is an isomorphism. Equivalently, since $\dim\Pi_{N-1}=N$, the family
$\mathcal S$ is $\omega$-admissible if and only if
\[
q\in\Pi_{N-1},
\quad
\mathcal H_i^\omega(q)=0,
\quad i=1,\ldots,N,
\implies q=0.
\]
Thus, admissibility is precisely the requirement that the weighted moments
over the intervals in $\mathcal S$ uniquely determine every polynomial in
$\Pi_{N-1}$. The following result shows that admissibility of the family
$\mathcal S$ is exactly the condition needed for the derivative moments to
determine a polynomial up to an additive constant.

\begin{proposition}\label{prop:admissible-derivative-kernel}
If $\mathcal S$ is $\omega$-admissible, then the restriction
\[
\mathcal D^{\omega}_{{\mathcal S}_{|{\Pi_{N}}}}
:
\Pi_{N}\to\RR^N
\]
is surjective and
\begin{equation}\label{eq:admissible-derivative-kernel}
\ker\left(\mathcal D^{\omega}_{{\mathcal S}_{|{\Pi_{N}}}}
\right)
=
\Pi_0.
\end{equation}
\end{proposition}
\begin{proof}
Let $\boldsymbol v\in\RR^N$ be arbitrary. Since
$\mathcal S$ is $\omega$-admissible, there exists a unique
$q\in\Pi_{N-1}$ such that
\[
\mathcal H_{\mathcal S}^{\omega}(q)=\boldsymbol v.
\]
Fix $x_0\in[-1,1]$ and define
\[
p(x)=\int_{x_0}^{x}q(t)\dd t.
\]
Then $p\in\Pi_N$ and $p'=q$. Therefore,
\begin{eqnarray*}
\mathcal D_{\mathcal S}^{\omega}(p)
&=&
\left[
\int_{s_1}p'(x)\omega(x)\dd x,
\ldots,
\int_{s_N}p'(x)\omega(x)\dd x
\right]^\top\\
&=&
\left[
\int_{s_1}q(x)\omega(x)\dd x,
\ldots,
\int_{s_N}q(x)\omega(x)\dd x
\right]^\top
=
\mathcal H_{\mathcal S}^{\omega}(q)
=
\boldsymbol v.
\end{eqnarray*}
Hence $\mathcal D^{\omega}_{{\mathcal S}_{|{\Pi_{N}}}}$ is surjective.

It remains to identify its kernel. Let
\[
p\in
\ker\left(\mathcal D^{\omega}_{{\mathcal S}_{|{\Pi_{N}}}}
\right).
\]
Then $p\in\Pi_N$ and
\[
\mathcal H_{\mathcal S}^{\omega}(p')
=
\mathcal D_{\mathcal S}^{\omega}(p)
=
\boldsymbol 0.
\]
Since $p'\in\Pi_{N-1}$ and
$\mathcal H^{\omega}_{{\mathcal S}_{|{\Pi_{N-1}}}}$ is injective, we
obtain $p'=0$. Hence $p$ is constant, and therefore
\[
\ker\left(\mathcal D^{\omega}_{{\mathcal S}_{|{\Pi_{N}}}}
\right)
\subseteq
\Pi_0.
\]
Conversely, every constant polynomial belongs to this kernel, because its
derivative vanishes identically. Thus
\[
\Pi_0
\subseteq
\ker\left(\mathcal D^{\omega}_{{\mathcal S}_{|{\Pi_{N}}}}
\right).
\]
Combining the two inclusions gives~\eqref{eq:admissible-derivative-kernel}.
\end{proof}

The previous result shows that, for an $\omega$-admissible family of
intervals, the derivative moments determine a polynomial only up to an
additive constant. Therefore, a single additional scalar condition is needed
to obtain unisolvence on $\Pi_N$.

Let
\[
\mathcal M:C([-1,1])\to\RR
\]
be a linear functional, and define the set of degrees of freedom
\[
\Sigma_{\mathcal M,\mathcal S}^{\omega}
=
\left\{
\mathcal M,
\mathcal D_1^\omega,\ldots,
\mathcal D_N^\omega
\right\}.
\]
The functional $\mathcal M$ plays the role of a normalization, fixing the
constant component which is not detected by the derivative moments.

\begin{theorem}\label{thm:admissible-unisolvence}
Let $\mathcal S=\left\{s_1,\ldots,s_N\right\}$
be a family of nondegenerate intervals in $[-1,1]$, and let
$\omega\in L^1(-1,1)$ satisfy~\eqref{eq:positive-weight}. Let
$\mathcal M$ be a linear functional on $C([-1,1])$. Then the
degrees of freedom
\[
\Sigma_{\mathcal M,\mathcal S}^{\omega}
=
\left\{
\mathcal M,
\mathcal D_1^\omega,\ldots,\mathcal D_N^\omega
\right\}
\]
are unisolvent on $\Pi_N$ if and only if $\mathcal S$ is
$\omega$-admissible and
\[
\mathcal M(1)\neq0.
\]
\end{theorem}

\begin{proof}
Assume first that $\mathcal S$ is $\omega$-admissible and that
$\mathcal M(1)\neq0$. Let $p\in\Pi_N$ satisfy
\[
\mathcal M(p)=0,
\quad
\mathcal D_i^\omega(p)=0,
\quad i=1,\ldots,N.
\]
In particular,
\[
\mathcal D_{\mathcal S}^{\omega}(p)=\boldsymbol 0.
\]
By Proposition~\ref{prop:admissible-derivative-kernel}, we have
\[
p\in
\ker\left(\mathcal D^{\omega}_{{\mathcal S}_{|{\Pi_N}}}\right)
=
\Pi_0.
\]
Hence $p=c$ for some $c\in\RR$. Since
\[
0=\mathcal M(p)=c\mathcal M(1),
\]
and $\mathcal M(1)\neq0$, it follows that $c=0$. Thus the common kernel of
the functionals in $\Sigma_{\mathcal M,\mathcal S}^{\omega}$ is trivial.
Since $\dim\Pi_N=N+1$, this proves unisolvence.

Conversely, assume that
$\Sigma_{\mathcal M,\mathcal S}^{\omega}$ is unisolvent on $\Pi_N$. Since
\[
\mathcal D_i^\omega(1)=0,
\quad i=1,\ldots,N,
\]
we must have $\mathcal M(1)\neq0$. Otherwise the nonzero polynomial $1$ would
belong to the common kernel. It remains to prove that $\mathcal S$ is $\omega$-admissible. Let
$q\in\Pi_{N-1}$ satisfy
\[
\mathcal H_i^\omega(q)=0,
\quad i=1,\ldots,N.
\]
Choose $p\in\Pi_N$ such that $p'=q$. Then
\[
\mathcal D_i^\omega(p)
=
\mathcal H_i^\omega(q)
=
0,
\quad i=1,\ldots,N.
\]
Set
\[
\widetilde p
=
p-\frac{\mathcal M(p)}{\mathcal M(1)}.
\]
Then
\[
\mathcal M\left(\widetilde p\right)=0,
\quad
\mathcal D_i^\omega\left(\widetilde p\right)=\mathcal D_i^\omega(p)=0,
\quad i=1,\ldots,N.
\]
By unisolvence, $\widetilde p=0$. Hence $p$ is constant, and therefore
\[
q=p'=0.
\]
Thus the restricted weighted moment map is injective on
$\Pi_{N-1}$. Since $\dim\Pi_{N-1}=N$, it is an isomorphism. Therefore
$\mathcal S$ is $\omega$-admissible.
\end{proof}
\begin{remark}\label{rem:weighted-mean-normalization}
A simple class of admissible normalizations is provided by weighted mean
functionals. Assume that $\mathcal S$ is $\omega$-admissible, and let
$\eta\in L^1(-1,1)$ satisfy
\begin{equation*}
\int_{-1}^{1}\eta(x)\dd x\neq0.
\end{equation*}
Define
\begin{equation*}
\mathcal M_\eta(f)
=
\int_{-1}^{1}f(x)\eta(x)\dd x,
\quad f\in C([-1,1]).
\end{equation*}
Then
\begin{equation*}
\mathcal M_\eta(1)
=
\int_{-1}^{1}\eta(x)\dd x
\neq0.
\end{equation*}
Therefore, by Theorem~\ref{thm:admissible-unisolvence}, the degrees of freedom
\begin{equation*}
\left\{
\mathcal M_\eta,
\mathcal D_1^\omega,\ldots,
\mathcal D_N^\omega
\right\}
\end{equation*}
are unisolvent on $\Pi_N$.
\end{remark}

\begin{remark}\label{rem:normalization}
Assume that $\mathcal S$ is $\omega$-admissible. For prescribed data
$v_1,\ldots,v_N\in\RR$, the surjectivity of
$\mathcal D^{\omega}_{{\mathcal S}_{|{\Pi_N}}}$ ensures that there exists
$p_0\in\Pi_N$ such that
\begin{equation*}
\mathcal D_i^\omega(p_0)=v_i,
\quad i=1,\ldots,N.
\end{equation*}
Moreover, if $p\in\Pi_N$ has the same derivative moments as $p_0$, then
\begin{equation*}
\mathcal D_{\mathcal S}^{\omega}\left(p-p_0\right)=\boldsymbol 0.
\end{equation*}
By Proposition~\ref{prop:admissible-derivative-kernel}, and since
$p-p_0\in\Pi_N$, this is equivalent to
\[
p-p_0\in\Pi_0.
\]
Hence every polynomial with the prescribed derivative
moments is of the form
\begin{equation*}
p=p_0+c,
\quad c\in\RR.
\end{equation*}
Thus the derivative moments determine the polynomial uniquely up to an
additive constant.

The normalization $\mathcal M(p)=m$ then gives
\begin{equation*}
\mathcal M\left(p_0\right)+c \mathcal M(1)=m.
\end{equation*}
If $\mathcal M(1)\neq0$, this equation determines $c$ uniquely, namely
\begin{equation*}
c=\frac{m-\mathcal M\left(p_0\right)}{\mathcal M(1)}.
\end{equation*}
Therefore, for arbitrary data $m,v_1,\ldots,v_N\in\RR$, the system
\begin{equation*}
\mathcal M(p)=m,
\quad
\mathcal D_i^\omega(p)=v_i,
\quad i=1,\ldots,N,
\end{equation*}
admits a unique solution in $\Pi_N$.
\end{remark}

\begin{remark}\label{rem:normalized-averages}
The functionals in~\eqref{eq:derivative-functional} may equivalently be
replaced by the normalized weighted averages
\begin{equation*}
\widehat{\mathcal D}_i^\omega(p)
=
\frac{1}{\mu_i^\omega}
\int_{s_i}p'(x)\omega(x)\dd x,
\quad
\mu_i^\omega=
\int_{s_i}\omega(x)\dd x.
\end{equation*}
Since $\omega>0$ almost everywhere and each interval $s_i$ is nondegenerate,
we have $\mu_i^\omega>0$. Therefore
\begin{equation*}
\widehat{\mathcal D}_i^\omega(p)=0
\quad\Longleftrightarrow\quad
\mathcal D_i^\omega(p)=0,
\quad i=1,\ldots,N.
\end{equation*}
Thus the normalized and unnormalized derivative moments have the same common
kernel. In particular, replacing $\mathcal D_i^\omega$ by
$\widehat{\mathcal D}_i^\omega$ does not affect the unisolvence result.
\end{remark}
We next give a matrix characterization of admissibility. Let
\[
\mathcal P=\left\{\varphi_0,\ldots,\varphi_{N-1}\right\}
\]
be any basis of $\Pi_{N-1}$, and define the corresponding weighted moment
matrix by
\begin{equation*}
G_{\mathcal S}^{\omega,\mathcal P}
=
\left[
\int_{s_i}\varphi_{k-1}(x)\omega(x)\dd x
\right]_{i,k=1}^{N}.
\end{equation*}

\begin{proposition}\label{prop:determinantal-admissibility}
The family $\mathcal S$ is $\omega$-admissible if and only if
\begin{equation}\label{eq:determinantal-admissibility}
\det\left(G_{\mathcal S}^{\omega,\mathcal P}\right)\neq0.
\end{equation}
\end{proposition}

\begin{proof}
The matrix $G_{\mathcal S}^{\omega,\mathcal P}$ is the matrix representation
of the restricted map
\[
\mathcal H^{\omega}_{{\mathcal S}_{|{\Pi_{N-1}}}}
:
\Pi_{N-1}\to\RR^N
\]
with respect to the basis $\mathcal P$ of $\Pi_{N-1}$ and the canonical basis
of $\RR^N$. Therefore,
$\mathcal H^{\omega}_{{\mathcal S}_{|{\Pi_{N-1}}}}$ is an isomorphism
if and only if $G_{\mathcal S}^{\omega,\mathcal P}$ is nonsingular. This
proves~\eqref{eq:determinantal-admissibility}.
\end{proof}

The determinant formulation also makes clear that admissibility is an open
property with respect to the endpoints. Thus, if the endpoints of an
$\omega$-admissible family are perturbed slightly, the resulting family remains
$\omega$-admissible.

\begin{proposition}\label{prop:openness-admissibility}
Let $\omega\in L^1(-1,1)$ satisfy~\eqref{eq:positive-weight}. Then the set of
$\omega$-admissible ordered families of $N$ nondegenerate intervals is open
with respect to the relative topology induced by
\[
\mathfrak C_N
=
\left\{
\left[a_1,b_1,\ldots,a_N,b_N\right]^{\top}\in[-1,1]^{2N}\,
:\,
a_i<b_i,\quad i=1,\ldots,N
\right\}.
\]
\end{proposition}

\begin{proof}
Fix a basis $\mathcal P=\{\varphi_0,\ldots,\varphi_{N-1}\}$ of $\Pi_{N-1}$. For
each $k=1,\ldots,N$, set
\[
F_k(x)=\int_{0}^{x}\varphi_{k-1}(t)\omega(t)\dd t,
\quad x\in[-1,1].
\]
Since $\varphi_{k-1}\omega\in L^1(-1,1)$, the function $F_k$ is continuous on
$[-1,1]$. Hence, if $s_i=\left[a_i,b_i\right]$, then
\[
\int_{s_i}\varphi_{k-1}(x)\omega(x)\dd x
=
F_k\left(b_i\right)-F_k\left(a_i\right).
\]
Each entry of the matrix $G_{\mathcal S}^{\omega,\mathcal P}$ therefore
depends continuously on the endpoints of the intervals. It follows that
\[
\left[a_1,b_1,\ldots,a_N,b_N\right]^{\top}
\to
\det\left(G_{\mathcal S}^{\omega,\mathcal P}\right)
\]
is continuous on the set of configurations with $a_i<b_i$,
$i=1,\ldots,N$.

By Proposition~\ref{prop:determinantal-admissibility}, the family
$\mathcal S$ is $\omega$-admissible exactly when this determinant is nonzero.
The conclusion follows because the nonzero set of a continuous function is
open.
\end{proof}

We now construct admissible families of intervals from a fixed grid. Let
\[
X_N=\left\{x_0,\ldots,x_N\right\}
\]
be a grid such that
\[
-1\leq x_0<x_1<\cdots<x_N\leq1.
\]
We denote by
\[
e_j=\left[x_{j-1},x_j\right],
\quad j=1,\ldots,N,
\]
the intervals associated with $X_N$.

\begin{proposition}\label{prop:elementary-intervals-admissible}
The family
\[
\mathcal E=\left\{e_1,\ldots,e_N\right\}
\]
is $\omega$-admissible.
\end{proposition}

\begin{proof}
By definition, we have to show that the restricted moment map
$\mathcal H^{\omega}_{{\mathcal E}_{|{\Pi_{N-1}}}}$ is an isomorphism. Since
$\dim\Pi_{N-1}=N$, it is enough to prove that
$\mathcal H_{\mathcal E}^{\omega}$ is injective. Let $q\in\Pi_{N-1}$ satisfy
\[
\mathcal H_i^\omega(q)=\int_{e_i}q(x)\omega(x)\dd x=0,
\quad i=1,\ldots,N.
\]
Assume, by contradiction, that $q$ is not identically zero. Fix
$i\in\{1,\ldots,N\}$. If $q$ had no zero in $\left(x_{i-1},x_i\right)$, then, by
continuity, $q$ would be either positive or negative throughout this interval.
Since $\omega>0$ almost everywhere, the product $q\omega$ has constant sign
almost everywhere on $\left(x_{i-1},x_i\right)$ and is not identically zero.
Therefore
\[
\mathcal H_i^\omega(q)=\int_{e_i}q(x)\omega(x)\dd x\neq0.
\]
This contradicts the vanishing of the moment over $e_i$. Hence, for each
$i=1,\ldots,N$, there exists a point
\[
\xi_i\in \left(x_{i-1},x_i\right)
\]
such that $q\left(\xi_i\right)=0$. Since each $\xi_i$ belongs to the interior of $e_i$, and the intervals
$e_1,\ldots,e_N$ have pairwise disjoint interiors, the points
$\xi_1,\ldots,\xi_N$ are distinct. Thus $q$ has at least $N$ distinct
zeros. This is impossible for a nonzero polynomial of degree at most $N-1$.
Therefore $q=0$, and $\mathcal H^{\omega}_{{\mathcal E}_{|{\Pi_{N-1}}}}$
is injective. Consequently, $\mathcal E$ is $\omega$-admissible.
\end{proof}

Let $\mathcal S=\{s_1,\ldots,s_N\}$ be a family of intervals whose endpoints
belong to $X_N$. Thus, for each $i=1,\ldots,N$, there exist integers
$k_i,r_i$ such that
\[
0\leq k_i<r_i\leq N
\]
and
\[
s_i=\left[x_{k_i},x_{r_i}\right].
\]
We associate with $\mathcal S$ the interval matrix
\begin{equation}\label{eq:interval-incidence-matrix}
A_{\mathcal S}
=
\left[A_{ij}\right]_{i,j=1}^{N},
\quad
A_{ij}
=
\begin{cases}
1, & k_i<j\leq r_i,\\
0, & \text{otherwise}.
\end{cases}
\end{equation}
Equivalently, $A_{ij}=1$ if and only if $e_j\subseteq s_i$.

\begin{theorem}\label{thm:incidence-admissibility}
The family $\mathcal S$ is $\omega$-admissible if and
only if
\begin{equation*}
\det\left(A_{\mathcal S}\right)\neq0.
\end{equation*}
\end{theorem}

\begin{proof}
For any $i=1,\ldots,N$, we have
\[
s_i=\left[x_{k_i},x_{r_i}\right]
=
\bigcup_{j=k_i+1}^{r_i} e_j.
\]
Hence, for any $q\in\Pi_{N-1}$, by additivity of the integral, we get
\[
\int_{s_i}q(x)\omega(x)\dd x
=
\sum_{j=k_i+1}^{r_i}
\int_{e_j}q(x)\omega(x)\dd x
=
\sum_{j=1}^{N} A_{ij}
\int_{e_j}q(x)\omega(x)\dd x,
\quad i=1,\ldots,N.
\]
Therefore, using~\eqref{eq:admissible-moment-map}, we can write these
relations in compact form as
\begin{equation}\label{eq:incidence-factorization}
\mathcal H_{\mathcal S}^{\omega}(q)
=
A_{\mathcal S}
\mathcal H_{\mathcal E}^{\omega}(q),
\quad q\in\Pi_{N-1}.
\end{equation}
By Proposition~\ref{prop:elementary-intervals-admissible}, the 
family $\mathcal E=\left\{e_1,\ldots,e_N\right\}$ is $\omega$-admissible. Thus
$\mathcal H^{\omega}_{{\mathcal E}_{|{\Pi_{N-1}}}}$ is an isomorphism.

Assume first that $\det\left(A_{\mathcal S}\right)\neq0$. Let
$q\in\Pi_{N-1}$ satisfy
\[
\mathcal H_{\mathcal S}^{\omega}(q)=\boldsymbol 0.
\]
By~\eqref{eq:incidence-factorization}, we get
\[
A_{\mathcal S}
\mathcal H_{\mathcal E}^{\omega}(q)
=
\boldsymbol 0.
\]
Since $A_{\mathcal S}$ is nonsingular, it follows that
\[
\mathcal H_{\mathcal E}^{\omega}(q)=\boldsymbol 0.
\]
Since $\mathcal H_{\mathcal E}^{\omega}$ is injective, we obtain $q=0$.
Therefore $\mathcal H^{\omega}_{{\mathcal S}_{|{\Pi_{N-1}}}}$ is injective.
Since $\dim\Pi_{N-1}=N$, the map
$\mathcal H^{\omega}_{{\mathcal S}_{|{\Pi_{N-1}}}}$ is an isomorphism, and
$\mathcal S$ is $\omega$-admissible.

Conversely, assume that $\det\left(A_{\mathcal S}\right)=0$. Then there exists a nonzero
vector $\boldsymbol c\in\RR^N$ such that
\[
A_{\mathcal S}\boldsymbol c=\boldsymbol 0.
\]
Since $\mathcal H_{\mathcal E}^{\omega}$ is an isomorphism, there exists a nonzero polynomial 
$q\in\Pi_{N-1}$ such that
\[
\mathcal H_{\mathcal E}^{\omega}(q)=\boldsymbol c\neq \boldsymbol{0}.
\]
On the other hand,
by~\eqref{eq:incidence-factorization}, we get
\[
\mathcal H_{\mathcal S}^{\omega}(q)
=
A_{\mathcal S}
\mathcal H_{\mathcal E}^{\omega}(q)
=
A_{\mathcal S}\boldsymbol c
=
\boldsymbol 0.
\]
Hence $\mathcal H_{\mathcal S}^{\omega}$ is not injective, and therefore
$\mathcal S$ is not $\omega$-admissible.
\end{proof}

We use the previous result to prove the admissibility of the two interval
families shown in Fig.~\ref{fig:sLsR}.
\begin{proposition}
    \label{cor:cumulative-admissible-families}
Let
\[
X_N=\left\{x_0,\ldots,x_N\right\},
\quad
-1\leq x_0<x_1<\cdots<x_N\leq1.
\]
Then the families
\[
\mathcal{S}_L=\left\{s_{i,L}=\left[x_0,x_i\right]\,:\, i=1,\ldots,N\right\},
\]
and
\[
\mathcal{S}_R=\left\{s_{i,R}=\left[x_{i-1},x_N\right]\,:\, i=1,\ldots,N\right\},
\]
are $\omega$-admissible for any weight $\omega\in L^1(-1,1)$ satisfying~\eqref{eq:positive-weight}.
\end{proposition}
\begin{proof}
We first consider the family
\[
\mathcal S_L=\left\{\left[x_0,x_i\right]\,:\, i=1,\ldots,N\right\}.
\]
For a fixed $i$, the interval $\left[x_0,x_i\right]$ is the union of the elementary
intervals
\[
e_1,\ldots,e_i.
\]
Hence the corresponding  interval matrix $A_{\mathcal{S}_L}$ satisfies
\[
\left[A_{\mathcal{S}_L}\right]_{ij}=1
\quad\Longleftrightarrow\quad
e_j\subset \left[x_0,x_i\right]
\quad\Longleftrightarrow\quad
j\leq i.
\]
Therefore,
\[
\left[A_{\mathcal{S}_L}\right]_{ij}
=
\begin{cases}
1, & j\leq i,\\
0, & j>i.
\end{cases}
\]
Thus $A_{\mathcal{S}_L}$ is lower triangular and all its diagonal entries are equal to $1$.
Consequently,
\[
\det\left(A_{\mathcal{S}_L}\right)=1.
\]
By Theorem~\ref{thm:incidence-admissibility}, the family
$\mathcal S_L$ is $\omega$-admissible. The argument for the family $\mathcal S_R$ is analogous. In this case the
corresponding matrix is upper triangular with all diagonal entries equal to
$1$, and hence has determinant equal to $1$. Therefore, again by
Theorem~\ref{thm:incidence-admissibility}, the family $\mathcal S_R$ is
$\omega$-admissible.
\end{proof}

\begin{remark}
The admissibility of the families
\[
\mathcal{S}_L=\left\{\left[x_0,x_i\right]\,:\, i=1,\ldots,N\right\},
\quad
\mathcal{S}_R=\left\{\left[x_{i-1},x_N\right]\,:\, i=1,\ldots,N\right\},
\]
in the framework of histopolation has already been proved in~\cite{Bruno:2024:PIO,Bruno:2025:OTC}. In Theorem~\ref{thm:incidence-admissibility}, we show that these families
belong to a larger class. Thus, this theorem extends these known unisolvence
results to a broader class of interval families.
\end{remark} 

\begin{figure}[ht]
\centering
\includegraphics[width=0.45\textwidth]{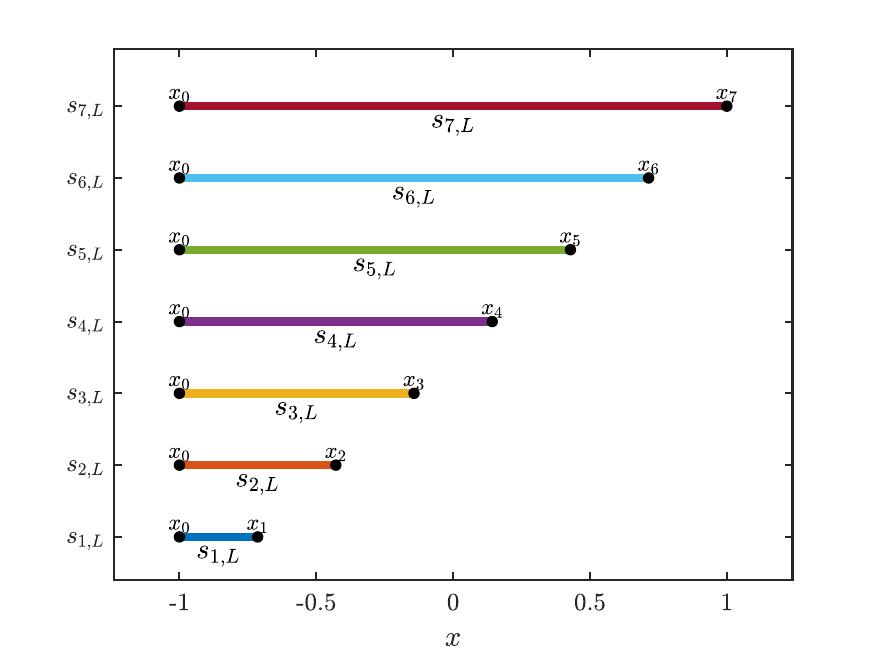}
\includegraphics[width=0.45\textwidth]{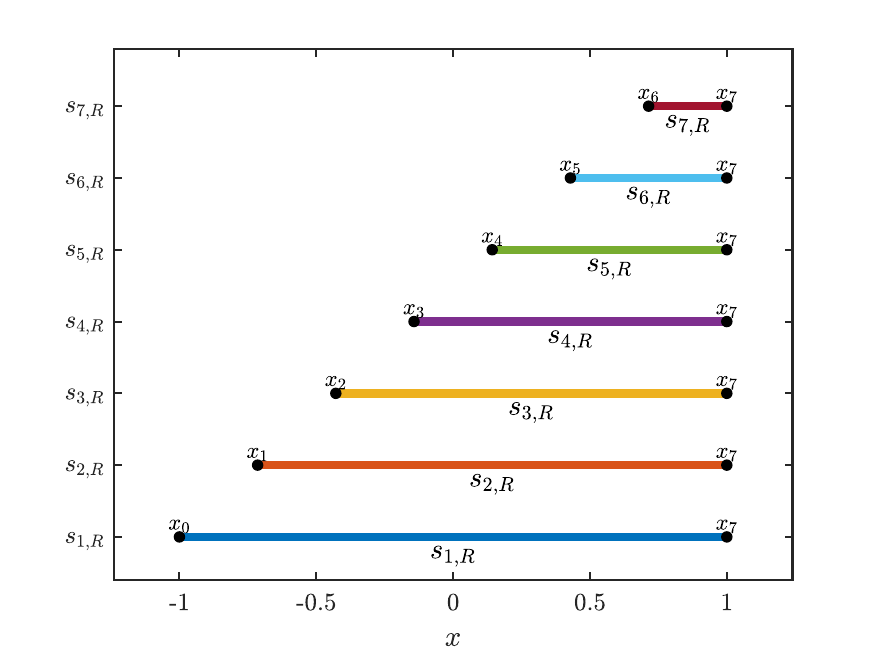}
\\[1ex]
\caption{Left: admissible family $\mathcal S_L$ formed by the intervals $s_{i,L}=\left[x_0,x_i\right]$, $i=1,\ldots,7$. Right: admissible family $\mathcal S_R$ formed by the intervals $s_{i,R}=\left[x_{i-1},x_7\right]$, $i=1,\ldots,7$.}
\label{fig:sLsR}
\end{figure}

The next result gives another admissible configuration, see Fig.~\ref{fig:admissible-interval-families}.

\begin{proposition}\label{cor:seeded-sliding-window-admissible}
Let $N\geq4$, and let
\[
X_N=\left\{x_0,\ldots,x_N\right\},
\quad
-1\leq x_0<x_1<\cdots<x_N\leq1.
\]
Let $\ell$ be an integer such that
\[
3\leq \ell\leq N-1.
\]
Define the family $\mathcal S_\ell=\left\{s_1,\ldots,s_N\right\}$ as
\[
s_i=\left[x_{i-1},x_{i+\ell-1}\right],
\quad i=1,\ldots,N-\ell+1,
\]
and
\[
s_{N-\ell+1+r}=\left[x_0,x_r\right],
\quad r=1,\ldots,\ell-1.
\]
Then $\mathcal S_\ell$ is $\omega$-admissible for any weight
$\omega\in L^1(-1,1)$ satisfying~\eqref{eq:positive-weight}.
\end{proposition}

\begin{proof}
By Theorem~\ref{thm:incidence-admissibility}, it is enough to prove that the
interval matrix $A_{\mathcal S_\ell}$ is nonsingular. We describe its rows.
For the intervals
\[
s_i=\left[x_{i-1},x_{i+\ell-1}\right],
\quad i=1,\ldots,N-\ell+1,
\]
we have
\[
s_i=e_i\cup e_{i+1}\cup\cdots\cup e_{i+\ell-1}.
\]
Hence the first $N-\ell+1$ rows of $A_{\mathcal S_\ell}$ contain $\ell$
consecutive entries equal to $1$, namely
\[
\left[A_{\mathcal S_\ell}\right]_{ij}=1
\quad\Longleftrightarrow\quad
i\leq j\leq i+\ell-1,
\quad i=1,\ldots,N-\ell+1.
\]
For the remaining intervals
\[
s_{N-\ell+1+r}=\left[x_0,x_r\right],
\quad r=1,\ldots,\ell-1,
\]
we have
\[
s_{N-\ell+1+r}=e_1\cup\cdots\cup e_r.
\]
Therefore the last $\ell-1$ rows satisfy
\[
\left[A_{\mathcal S_\ell}\right]_{N-\ell+1+r,j}=1
\quad\Longleftrightarrow\quad
1\leq j\leq r,
\quad r=1,\ldots,\ell-1.
\]
Let
\[
\boldsymbol c=\left[c_1,\ldots,c_N\right]^\top\in\RR^N
\]
satisfy
\[
A_{\mathcal S_\ell}\boldsymbol c=\boldsymbol 0.
\]
From the last $\ell-1$ rows we get
\[
\sum_{j=1}^{r}c_j=0,
\quad r=1,\ldots,\ell-1.
\]
Taking $r=1$ gives $c_1=0$. Subtracting two consecutive identities gives
successively
\begin{equation}\label{cond1a}
    c_2=\cdots=c_{\ell-1}=0.
\end{equation}
From the first $N-\ell+1$ rows we get
\[
\sum_{j=i}^{i+\ell-1}c_j=0,
\quad i=1,\ldots,N-\ell+1.
\]
For $i=1$, this identity reads
\[
c_1+\cdots+c_{\ell-1}+c_\ell=0.
\]
By~\eqref{cond1a}, it follows that $c_\ell=0$. For $i=2$, we get
\[
c_2+\cdots+c_\ell+c_{\ell+1}=0,
\]
and therefore $c_{\ell+1}=0$. Repeating this argument gives
\[
c_\ell=c_{\ell+1}=\cdots=c_N=0.
\]
Together with the previous identities, this yields
\[
\boldsymbol c=\boldsymbol 0.
\]
Hence $A_{\mathcal S_\ell}$ has trivial kernel, and is therefore nonsingular.
By Theorem~\ref{thm:incidence-admissibility}, the family $\mathcal S_\ell$ is
$\omega$-admissible.
\end{proof}

The following example shows that the admissibility criterion established in
Theorem~\ref{thm:incidence-admissibility} goes beyond the two families of
Proposition~\ref{cor:cumulative-admissible-families}. For $N\geq3$, it provides
interval families which are neither nested nor pairwise disjoint in their
interiors, see Fig.~\ref{fig:admissible-interval-families}.

\begin{proposition}\label{prop:overlapping-family-admissible}
Let $N\geq 2$, and let
\[
X_N=\left\{x_0,\ldots,x_N\right\},
\quad
-1\leq x_0<x_1<\cdots<x_N\leq1.
\]
Define
\[
s_i=\left[x_{i-1},x_{i+1}\right],
\quad i=1,\ldots,N-1, \quad s_N=\left[x_0,x_1\right].
\]
Then the family
\[
\mathcal S=\left\{s_1,\ldots,s_N\right\}
\]
is $\omega$-admissible for any weight $\omega\in L^1(-1,1)$
satisfying~\eqref{eq:positive-weight}.
\end{proposition}

\begin{proof}
By Theorem~\ref{thm:incidence-admissibility}, it is enough to prove that the interval 
matrix $A_{\mathcal S}$ defined in~\eqref{eq:interval-incidence-matrix} is
nonsingular. For $i=1,\ldots,N-1$, we have
\[
s_i=\left[x_{i-1},x_{i+1}\right]=e_i\cup e_{i+1},
\]
whereas
\[
s_N=\left[x_0,x_1\right]=e_1.
\]
Hence, by the definition of $A_{\mathcal S}$, we have
\[
A_{\mathcal S}
=
\left[
\begin{array}{cccccc}
1 & 1 & 0 & \cdots & 0 & 0\\
0 & 1 & 1 & \cdots & 0 & 0\\
0 & 0 & 1 & \ddots & 0 & 0\\
\vdots & \vdots & \vdots & \ddots & 1 & 0\\
0 & 0 & 0 & \cdots & 1 & 1\\
1 & 0 & 0 & \cdots & 0 & 0
\end{array}
\right].
\]
We now show that this matrix has trivial kernel. Let
\[
\boldsymbol c=\left[c_1,\ldots,c_N\right]^\top\in\RR^N
\]
satisfy
\[
A_{\mathcal S}\boldsymbol c=\boldsymbol 0.
\]
The last row gives $c_1=0$. On the other hand, the first $N-1$ rows give
\[
c_i+c_{i+1}=0,
\quad i=1,\ldots,N-1.
\]
Starting from $c_1=0$, these identities imply successively
\[
c_2=c_3=\cdots=c_N=0.
\]
Therefore $\boldsymbol c=\boldsymbol 0$, and $A_{\mathcal S}$ is
nonsingular. By Theorem~\ref{thm:incidence-admissibility}, the family
$\mathcal S$ is $\omega$-admissible.
\end{proof}

\begin{figure}[ht]
\centering
\includegraphics[width=0.45\textwidth]{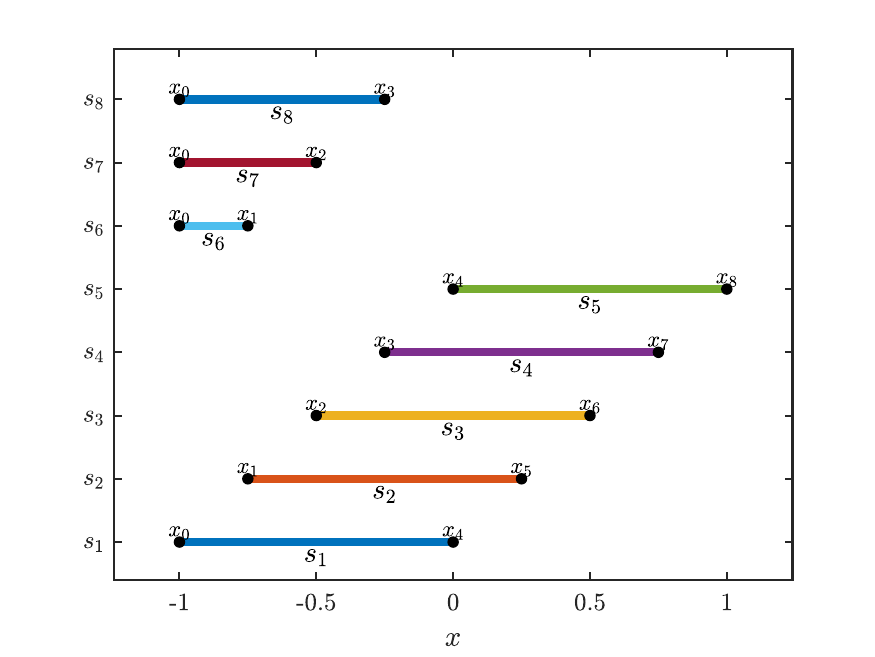}
\includegraphics[width=0.45\textwidth]{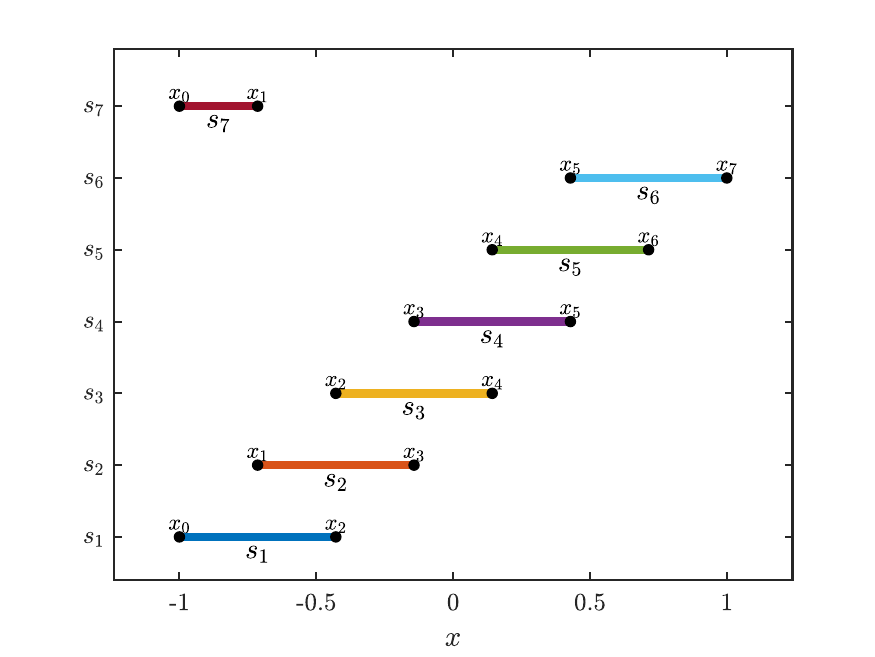}
\\[1ex]
\caption{Left: admissible family $\mathcal S_\ell$ obtained from intervals of fixed length, with $N=8$ and $\ell=4$. Right: overlapping admissible family with intervals $s_i=\left[x_{i-1},x_{i+1}\right]$, $i=1,\ldots,6$, and $s_7=[x_0,x_1]$.}
\label{fig:admissible-interval-families}
\end{figure}

In the next examples, we give two explicit configurations, for $N=2$ and
$N=3$, in which admissibility can be checked directly by computing the
associated interval matrix~\eqref{eq:interval-incidence-matrix}.

\begin{example}\label{ex:admissible-matrix-N2}
Let $N=2$, and consider the grid
\[
X_2=\{-1,0,1\}.
\]
Let
\[
e_1=[-1,0],
\quad
e_2=[0,1],
\]
and define
\[
s_1=[-1,1],
\quad
s_2=[-1,0].
\]
Then
\[
s_1=e_1\cup e_2,
\quad
s_2=e_1.
\]
Hence the associated matrix~\eqref{eq:interval-incidence-matrix} is
\[
A_{\mathcal S}
=
\begin{bmatrix}
1 & 1\\
1 & 0
\end{bmatrix}.
\]
A direct computation gives
\[
\det\left(A_{\mathcal S}\right)=-1\neq0.
\]
By Theorem~\ref{thm:incidence-admissibility}, the family
$\mathcal S=\left\{s_1,s_2\right\}$ is $\omega$-admissible for any weight
$\omega\in L^1(-1,1)$ satisfying~\eqref{eq:positive-weight}. 
\end{example}

\begin{example}\label{ex:admissible-matrix-N3}
Let $N=3$, and consider the grid
\[
X_3=\left\{-1,-\frac{1}{3},\frac{1}{3},1\right\}.
\]
Let
\[
e_1=\left[-1,-\frac{1}{3}\right],
\quad
e_2=\left[-\frac{1}{3},\frac{1}{3}\right],
\quad
e_3=\left[\frac{1}{3},1\right].
\]
We define
\[
s_1=\left[-1,\frac{1}{3}\right],
\quad
s_2=\left[-\frac{1}{3},1\right],
\quad
s_3=[-1,1].
\]
Equivalently,
\[
s_1=e_1\cup e_2,
\quad
s_2=e_2\cup e_3,
\quad
s_3=e_1\cup e_2\cup e_3.
\]
Hence the associated matrix~\eqref{eq:interval-incidence-matrix} is
\[
A_{\mathcal S}
=
\begin{bmatrix}
1 & 1 & 0\\
0 & 1 & 1\\
1 & 1 & 1
\end{bmatrix}.
\]
Thus
\[
\det\left(A_{\mathcal S}\right)=1\neq0.
\]
By Theorem~\ref{thm:incidence-admissibility}, the family $\mathcal S$ is
$\omega$-admissible for every weight $\omega\in L^1(-1,1)$
satisfying~\eqref{eq:positive-weight}.
\end{example}

The next result shows that, on consecutive cells and for the constant weight,
the derivative moment conditions are exactly equivalent to classical nodal
interpolation. In this particular setting, the weighted derivative construction
therefore reduces to the Lagrange interpolating polynomial.

\begin{proposition}\label{prop:lagrange-special-case}
Let
\[
X_N=\left\{x_0,\ldots,x_N\right\},
\quad
-1\leq x_0<x_1<\cdots<x_N\leq1,
\]
and set
\begin{equation*}
s_i=\left[x_{i-1},x_i\right],
\quad i=1,\ldots,N.
\end{equation*}
Assume that $\omega=1$, and define
\begin{equation*}
\mathcal M(p)=p\left(x_0\right),
\quad p\in\Pi_N.
\end{equation*}
Let $f\in C^1\left(\left[x_0,x_N\right]\right)$, and set
\begin{equation*}
\mathcal D_i^1(f)
=
\int_{x_{i-1}}^{x_i}f'(x)\dd x,
\quad i=1,\ldots,N.
\end{equation*}
Then a polynomial $p\in\Pi_N$ satisfies
\begin{equation}\label{eq:lagrange-equivalent-data}
\mathcal M(p)=f\left(x_0\right),
\quad
\mathcal D_i^1(p)=\mathcal D_i^1(f),
\quad i=1,\ldots,N,
\end{equation}
if and only if
\begin{equation}\label{eq:lagrange-data}
p\left(x_i\right)=f\left(x_i\right),
\quad i=0,\ldots,N.
\end{equation}
Consequently, the unique polynomial satisfying~\eqref{eq:lagrange-equivalent-data}
is the Lagrange interpolating polynomial of $f$ at the nodes
$x_0,\ldots,x_N$.
\end{proposition}

\begin{proof}
Since $\omega=1$, the fundamental theorem of calculus gives
\begin{equation}\label{eq:cell-endpoint-differences}
\mathcal D_i^1(p)
=
\int_{x_{i-1}}^{x_i}p'(x)\dd x
=
p\left(x_i\right)-p\left(x_{i-1}\right),
\quad i=1,\ldots,N.
\end{equation}
Similarly,
\begin{equation}\label{eq:cell-endpoint-differences-f}
\mathcal D_i^1(f)
=
\int_{x_{i-1}}^{x_i}f'(x)\dd x
=
f\left(x_i\right)-f\left(x_{i-1}\right),
\quad i=1,\ldots,N.
\end{equation}
Assume first that $p$ satisfies~\eqref{eq:lagrange-equivalent-data}. From
\eqref{eq:cell-endpoint-differences} and
\eqref{eq:cell-endpoint-differences-f}, the derivative moment conditions give
\begin{equation*}
p\left(x_i\right)-p\left(x_{i-1}\right)
=
f\left(x_i\right)-f\left(x_{i-1}\right),
\quad i=1,\ldots,N.
\end{equation*}
Moreover, the normalization in~\eqref{eq:lagrange-equivalent-data} gives
\begin{equation*}
p\left(x_0\right)=f\left(x_0\right).
\end{equation*}
Therefore
\begin{equation*}
p\left(x_1\right)
=
p\left(x_0\right)+f\left(x_1\right)-f\left(x_0\right)
=
f\left(x_1\right).
\end{equation*}
Repeating the same argument on $s_2,\ldots,s_N$, we obtain
\begin{equation*}
p\left(x_i\right)=f\left(x_i\right),
\quad i=0,\ldots,N.
\end{equation*}
Hence~\eqref{eq:lagrange-data} is satisfied.

Conversely, suppose that~\eqref{eq:lagrange-data} is satisfied. Then
\begin{equation*}
\mathcal M(p)=p\left(x_0\right)=f\left(x_0\right).
\end{equation*}
Moreover, for any  $i=1,\ldots,N$, equations~\eqref{eq:cell-endpoint-differences}
and~\eqref{eq:cell-endpoint-differences-f} yield
\begin{equation*}
\mathcal D_i^1(p)
=
p\left(x_i\right)-p\left(x_{i-1}\right)
=
f\left(x_i\right)-f\left(x_{i-1}\right)
=
\mathcal D_i^1(f).
\end{equation*}
Thus~\eqref{eq:lagrange-equivalent-data} follows.
\end{proof}

Assume that $\mathcal S$ is $\omega$-admissible and that
$\mathcal M(1)\neq0$. We now introduce a Lebesgue-type constant associated with the degrees of
freedom
\[
\Sigma_{\mathcal M,\mathcal S}^{\omega}
=
\left\{
\mathcal M,\mathcal D_1^\omega,\ldots,\mathcal D_N^\omega
\right\}.
\]
 By Theorem~\ref{thm:admissible-unisolvence}, there
exists a unique basis
\[
\left\{\psi_0,\psi_1,\ldots,\psi_N\right\}
\]
of $\Pi_N$ satisfying
\[
\mathcal M\left(\psi_0\right)=1,
\quad
\mathcal D_i^\omega\left(\psi_0\right)=0,
\quad i=1,\ldots,N,
\]
and, for $j=1,\ldots,N$,
\[
\mathcal M\left(\psi_j\right)=0,
\quad
\mathcal D_i^\omega\left(\psi_j\right)
=
\begin{cases}
1, & i=j,\\
0, & i\neq j,
\end{cases}
\quad i=1,\ldots,N.
\]
We define the associated Lebesgue-type constant by
\[
\Lambda_{\mathcal M,\mathcal S}^{\omega}
=
\max_{x\in[-1,1]}
\sum_{j=0}^{N}\left|\psi_j(x)\right|.
\]

\begin{proposition}\label{prop:lebesgue-error-estimate}
Let $f\in C^1([-1,1])$, and let
\[
K_{\mathcal M,\mathcal S}^{\omega}:f\in C^1([-1,1])\to K_{\mathcal M,\mathcal S}^{\omega}[f]\in \Pi_N,
\]
be defined by
\[
\mathcal M\left(K_{\mathcal M,\mathcal S}^{\omega}[f]\right)
=
\mathcal M(f),
\quad
\mathcal D_i^\omega\left(K_{\mathcal M,\mathcal S}^{\omega}[f]\right)
=
\mathcal D_i^\omega(f),
\quad i=1,\ldots,N.
\]
Let $q_N^\ast\in\Pi_N$ be such that
\[
\|f-q_N^\ast\|_\infty
=
\inf_{q\in\Pi_N}\|f-q\|_\infty.
\]
Then
\begin{equation}\label{eqsr}
    \left\|
f-K_{\mathcal M,\mathcal S}^{\omega}[f]
\right\|_{\infty}
\leq
\|f-q_N^\ast\|_{\infty}
+
\Lambda_{\mathcal M,\mathcal S}^{\omega}
\max\left\{
\left|\mathcal M\left(f-q_N^\ast\right)\right|,
\max_{1\leq i\leq N}
\left|\mathcal D_i^\omega\left(f-q_N^\ast\right)\right|
\right\}.
\end{equation}
\end{proposition}

\begin{proof}
Let $q\in\Pi_N$. By uniqueness, we have
\[
K_{\mathcal M,\mathcal S}^{\omega}[q]=q.
\]
Since $K_{\mathcal M,\mathcal S}^{\omega}$ is linear, we get
\[
f-K_{\mathcal M,\mathcal S}^{\omega}[f]
=
(f-q)-K_{\mathcal M,\mathcal S}^{\omega}[f-q].
\]
Hence
\begin{equation}\label{eq:first-lebesgue-estimate}
\left\|
f-K_{\mathcal M,\mathcal S}^{\omega}[f]
\right\|_\infty
\leq
\|f-q\|_\infty
+
\left\|
K_{\mathcal M,\mathcal S}^{\omega}[f-q]
\right\|_\infty .
\end{equation}
By the definition of the basis $\{\psi_0,\psi_1,\ldots,\psi_N\}$, we have
\[
K_{\mathcal M,\mathcal S}^{\omega}[f-q]
=
\mathcal M(f-q)\psi_0
+
\sum_{i=1}^{N}\mathcal D_i^\omega(f-q)\psi_i .
\]
Therefore,
\begin{equation}\label{eq:second-lebesgue-estimate}
\left\|
K_{\mathcal M,\mathcal S}^{\omega}[f-q]
\right\|_\infty
\leq
\Lambda_{\mathcal M,\mathcal S}^{\omega}
\max\left\{
|\mathcal M(f-q)|,
\max_{1\leq i\leq N}
|\mathcal D_i^\omega(f-q)|
\right\}.
\end{equation}
Combining~\eqref{eq:first-lebesgue-estimate} and
\eqref{eq:second-lebesgue-estimate}, we obtain
\begin{equation*}
\left\|
f-K_{\mathcal M,\mathcal S}^{\omega}[f]
\right\|_{\infty}
\leq
\|f-q\|_{\infty}+
\Lambda_{\mathcal M,\mathcal S}^{\omega}
\max\left\{
|\mathcal M(f-q)|,
\max_{1\leq i\leq N}
|\mathcal D_i^\omega(f-q)|
\right\}.
\end{equation*}
Since this estimate holds for any $q\in\Pi_N$, we obtain~\eqref{eqsr}.
\end{proof}

\begin{remark}
If, in addition, $\mathcal M$ is continuous with respect to the uniform norm,
then, for any $f\in C^1([-1,1])$, we have
\[
\left|\mathcal M\left(f-q_N^\ast\right)\right|
\leq
\|\mathcal M\|\left\|f-q_N^\ast\right\|_\infty .
\]
Moreover, since $\omega\in L^1(-1,1)$, the derivative moments satisfy
\[
\left|\mathcal D_i^\omega\left(f-q_N^\ast\right)\right|
\leq
\|\omega\|_{L^1\left(s_i\right)}
\left\|\left(f-q_N^\ast\right)'\right\|_\infty,
\quad i=1,\ldots,N.
\]
In conclusion, under this additional assumption, the bound~\eqref{eqsr} is
explicit in terms of $\Lambda_{\mathcal M,\mathcal S}^{\omega}$,
$\left\|f-q_N^\ast\right\|_\infty$, and
$\left\|\left(f-q_N^\ast\right)'\right\|_\infty$.
\end{remark}

\section{Jacobi-weighted derivative histopolation}
\label{sec3}

In this section, we study the weighted derivative histopolation problem for
Jacobi weights. Let $\alpha,\beta>-1$, and define
\begin{equation}\label{eq:jacobi-weight}
\omega_{\alpha,\beta}(x)
=
(1-x)^\alpha(1+x)^\beta,
\quad x\in(-1,1).
\end{equation}
Then $\omega_{\alpha,\beta}\in L^1(-1,1)$ and
$\omega_{\alpha,\beta}>0$ almost everywhere on $(-1,1)$. Let
$\mathcal S=\left\{s_1,\ldots,s_N\right\}$ be an
$\omega_{\alpha,\beta}$-admissible family of nondegenerate intervals. We denote by
\begin{equation*}
P_0^{(\alpha,\beta)},\ldots,P_N^{(\alpha,\beta)}
\end{equation*}
the Jacobi polynomials up to degree $N$ in the standard normalization, so that
$P_0^{(\alpha,\beta)}=1$. As a normalization, we define
\begin{equation}\label{eq:jacobi-mean}
\mathcal M_{\alpha,\beta}(f)
=
\int_{-1}^{1}
f(x)\omega_{\alpha,\beta}(x)\dd x,
\quad f\in C([-1,1]).
\end{equation}
On the constant polynomial, we have
\begin{equation*}
\mathcal M_{\alpha,\beta}(1)
=
\int_{-1}^{1}\omega_{\alpha,\beta}(x)\dd x
=
2^{\alpha+\beta+1}
\frac{\Gamma(\alpha+1)\Gamma(\beta+1)}
{\Gamma(\alpha+\beta+2)}
>0,
\end{equation*}
where $\Gamma(\cdot)$ denotes the Gamma function~\cite{Abramowitz:1948:HOM,Milovanovic:1997:OPS}. Hence, by Theorem~\ref{thm:admissible-unisolvence}, the degrees of freedom
\begin{equation*}
\left\{
\mathcal M_{\alpha,\beta},
\mathcal D_1^{\omega_{\alpha,\beta}},
\ldots,
\mathcal D_N^{\omega_{\alpha,\beta}}
\right\}
\end{equation*}
are unisolvent on $\Pi_N$. We denote by $D_N^{(\alpha,\beta)}$ the matrix of
$\mathcal D_{\mathcal S}^{\omega_{\alpha,\beta}}$, defined
in~\eqref{eq:admissible-derivative-map}, restricted to
\[
\operatorname{span}\left\{
P_1^{(\alpha,\beta)},\ldots,P_N^{(\alpha,\beta)}
\right\},
\]
with respect to the basis
$\left\{P_1^{(\alpha,\beta)},\ldots,P_N^{(\alpha,\beta)}\right\}$ and the
canonical basis of $\RR^N$. Thus
\begin{equation}\label{eq:jacobi-derivative-matrix}
\left[D_N^{(\alpha,\beta)}\right]_{ik}
=
\mathcal D_i^{\omega_{\alpha,\beta}}
\left(P_k^{(\alpha,\beta)}\right)
=
\int_{s_i}
\frac{\dd}{\dd x}P_k^{(\alpha,\beta)}(x)
\omega_{\alpha,\beta}(x)\dd x,
\quad i,k=1,\ldots,N.
\end{equation}
We also introduce the Jacobi data map
\begin{equation*}
\mathcal J_{\mathcal S}^{(\alpha,\beta)}:
p\in\Pi_N
\to
\left[
\mathcal M_{\alpha,\beta}(p),
\mathcal D_1^{\omega_{\alpha,\beta}}(p),
\ldots,
\mathcal D_N^{\omega_{\alpha,\beta}}(p)
\right]^\top
\in\RR^{N+1}.
\end{equation*}

\begin{proposition}\label{prop:jacobi-block-form}
Let
$\mathcal S=\left\{s_1,\ldots,s_N\right\}$ be an
$\omega_{\alpha,\beta}$-admissible family of nondegenerate intervals. Let
$H_N^{(\alpha,\beta)}$ be the matrix representation of the map
$\mathcal J_{\mathcal S}^{(\alpha,\beta)}$ with respect to the Jacobi basis
\[
\mathcal B_N^{(\alpha,\beta)}
=
\left\{
P_0^{(\alpha,\beta)},\ldots,
P_N^{(\alpha,\beta)}
\right\}
\]
of $\Pi_N$ and the canonical basis of $\RR^{N+1}$. Then
\begin{equation}\label{eq:jacobi-block-form}
H_N^{(\alpha,\beta)}
=
\begin{bmatrix}
h_0^{(\alpha,\beta)} & \boldsymbol 0^\top\\[2mm]
\boldsymbol 0 & D_N^{(\alpha,\beta)}
\end{bmatrix},
\end{equation}
where
\begin{equation*}
h_0^{(\alpha,\beta)}
=
\mathcal M_{\alpha,\beta}(1)
=
\int_{-1}^{1}\omega_{\alpha,\beta}(x)\dd x.
\end{equation*}
In particular, $H_N^{(\alpha,\beta)}$ and $D_N^{(\alpha,\beta)}$ are
nonsingular.
\end{proposition}

\begin{proof}
The first row of $H_N^{(\alpha,\beta)}$ is obtained by applying
$\mathcal M_{\alpha,\beta}$ to the Jacobi basis. Since
$P_0^{(\alpha,\beta)}=1$, its first entry is
\begin{equation*}
\mathcal M_{\alpha,\beta}
\left(P_0^{(\alpha,\beta)}\right)
=
\mathcal M_{\alpha,\beta}(1)
=
h_0^{(\alpha,\beta)}.
\end{equation*}
For $k=1,\ldots,N$, the orthogonality of the Jacobi polynomials gives
\begin{equation*}
\mathcal M_{\alpha,\beta}
\left(P_k^{(\alpha,\beta)}\right)
=
\int_{-1}^{1}
P_k^{(\alpha,\beta)}(x)\omega_{\alpha,\beta}(x)\dd x
=
0.
\end{equation*}
Thus the first row is
\begin{equation*}
\left[
h_0^{(\alpha,\beta)},\boldsymbol 0^\top
\right].
\end{equation*}
On the other hand, since $P_0^{(\alpha,\beta)}=1$, we have
\begin{equation*}
\frac{\dd}{\dd x}P_0^{(\alpha,\beta)}(x)=0.
\end{equation*}
Hence
\begin{equation*}
\mathcal D_i^{\omega_{\alpha,\beta}}
\left(P_0^{(\alpha,\beta)}\right)
=
0,
\quad i=1,\ldots,N.
\end{equation*}
Therefore the first column below $h_0^{(\alpha,\beta)}$ vanishes. Finally, for $i,k=1,\ldots,N$, the remaining entries are 
\begin{equation*}
\mathcal D_i^{\omega_{\alpha,\beta}}
\left(P_k^{(\alpha,\beta)}\right)
=
\left[D_N^{(\alpha,\beta)}\right]_{ik}.
\end{equation*}
This proves~\eqref{eq:jacobi-block-form}. By Theorem~\ref{thm:admissible-unisolvence}, the degrees of freedom defining
$\mathcal J_{\mathcal S}^{(\alpha,\beta)}$ are unisolvent on $\Pi_N$.
Therefore $H_N^{(\alpha,\beta)}$ is nonsingular. Since
$h_0^{(\alpha,\beta)}>0$, the block form yields
\begin{equation*}
0
\neq
\det\left(H_N^{(\alpha,\beta)}\right)
=
h_0^{(\alpha,\beta)}
\det\left(D_N^{(\alpha,\beta)}\right).
\end{equation*}
Thus $\det\left(D_N^{(\alpha,\beta)}\right)\neq0$, and
$D_N^{(\alpha,\beta)}$ is nonsingular.
\end{proof}

\begin{remark}\label{rem-mixed-jacobi-block-form}
The block structure in Proposition~\ref{prop:jacobi-block-form} only uses the
fact that the first basis function is constant and that the remaining basis
functions have zero mean with respect to the normalization functional~\eqref{eq:jacobi-mean}. This
observation will be used below in situations where the parameters defining the
orthogonal basis and the parameters defining the derivative weight are not the
same. More precisely, let $\alpha,\beta,\gamma,\delta>-1$, and assume that
$\mathcal S$ is $\omega_{\gamma,\delta}$-admissible. Define
\begin{equation}\label{Jacobi_Map}
    \mathcal J_{\mathcal S}^{(\alpha,\beta;\gamma,\delta)}
:p\in \Pi_N
\to
\left[
\mathcal M_{\alpha,\beta}(p),
\mathcal D_1^{\omega_{\gamma,\delta}}(p),
\ldots,
\mathcal D_N^{\omega_{\gamma,\delta}}(p)
\right]^\top\in \mathbb{R}^{N+1}.
\end{equation}
Let $D_N^{(\alpha,\beta;\gamma,\delta)}$ be the reduced derivative block
\[
\left[D_N^{(\alpha,\beta;\gamma,\delta)}\right]_{ik}
=
\mathcal D_i^{\omega_{\gamma,\delta}}
\left(P_k^{(\alpha,\beta)}\right)
=
\int_{s_i}
\frac{\dd}{\dd x}P_k^{(\alpha,\beta)}(x)
\omega_{\gamma,\delta}(x)\dd x,
\quad i,k=1,\ldots,N.
\]
Then the matrix representation of
$\mathcal J_{\mathcal S}^{(\alpha,\beta;\gamma,\delta)}$ with respect to the
Jacobi basis
\[
\left\{
P_0^{(\alpha,\beta)},\ldots,P_N^{(\alpha,\beta)}
\right\}
\]
and the canonical basis of $\RR^{N+1}$ is
\[
\begin{bmatrix}
h_0^{(\alpha,\beta)} & \boldsymbol 0^\top\\[2mm]
\boldsymbol 0 & D_N^{(\alpha,\beta;\gamma,\delta)}
\end{bmatrix}.
\]
Indeed, the first row follows from the orthogonality of
$P_1^{(\alpha,\beta)},\ldots,P_N^{(\alpha,\beta)}$ with respect to
$\omega_{\alpha,\beta}$, while the first column below
$h_0^{(\alpha,\beta)}$ vanishes because $P_0^{(\alpha,\beta)}=1$.
Since $\mathcal S$ is $\omega_{\gamma,\delta}$-admissible and
$h_0^{(\alpha,\beta)}=\mathcal M_{\alpha,\beta}(1)>0$, the same argument as in
Proposition~\ref{prop:jacobi-block-form} shows that
$D_N^{(\alpha,\beta;\gamma,\delta)}$ is nonsingular.
\end{remark}

We next derive a useful decomposition of the reduced derivative block. We use
the Jacobi differentiation formula
\begin{equation}\label{eq:jacobi-derivative-formula}
\frac{\dd}{\dd x}P_k^{(\alpha,\beta)}(x)
=
\frac{k+\alpha+\beta+1}{2}
P_{k-1}^{(\alpha+1,\beta+1)}(x),
\quad k\geq1,
\end{equation}
where the Jacobi polynomials are taken with the standard normalization. For
$i,k=1,\ldots,N$, define the two shifted moment matrices
\begin{eqnarray*}
\left[A_N^{(\alpha,\beta)}\right]_{ik}
&=&
\int_{s_i}
P_{k-1}^{(\alpha+1,\beta+1)}(x)
\omega_{\alpha+1,\beta}(x)\dd x, 
\\
\left[B_N^{(\alpha,\beta)}\right]_{ik}
&=&
\int_{s_i}
P_{k-1}^{(\alpha+1,\beta+1)}(x)
\omega_{\alpha,\beta+1}(x)\dd x.
\end{eqnarray*}
Finally, set
\begin{equation*}
C_N^{(\alpha,\beta)}
=
\operatorname{diag}
\left(
\frac{k+\alpha+\beta+1}{2}
\right)_{k=1}^{N}.
\end{equation*}

\begin{theorem}
\label{thm:endpoint-shifted-decomposition}
The reduced Jacobi derivative block satisfies
\begin{equation}\label{eq:endpoint-shifted-decomposition}
D_N^{(\alpha,\beta)}
=
\frac{1}{2}
\left(A_N^{(\alpha,\beta)}+B_N^{(\alpha,\beta)}\right)
C_N^{(\alpha,\beta)}.
\end{equation}
\end{theorem}

\begin{proof}
Substituting~\eqref{eq:jacobi-derivative-formula} into
\eqref{eq:jacobi-derivative-matrix}, we obtain
\begin{equation}\label{eq:jacobi-D-first-step}
\left[D_N^{(\alpha,\beta)}\right]_{ik}
=
\frac{k+\alpha+\beta+1}{2}
\int_{s_i}
P_{k-1}^{(\alpha+1,\beta+1)}(x)
\omega_{\alpha,\beta}(x)\dd x, \quad i,k=1,\ldots,N.
\end{equation}
Moreover, for $x\in(-1,1)$, we have
\[
\omega_{\alpha+1,\beta}(x)
=
(1-x)\omega_{\alpha,\beta}(x),
\quad
\omega_{\alpha,\beta+1}(x)
=
(1+x)\omega_{\alpha,\beta}(x).
\]
Hence
\begin{equation}\label{eq:jacobi-weight-average}
\omega_{\alpha,\beta}(x)
=
\frac{1}{2}
\left(
\omega_{\alpha+1,\beta}(x)
+
\omega_{\alpha,\beta+1}(x)
\right).
\end{equation}
Using~\eqref{eq:jacobi-weight-average} in
\eqref{eq:jacobi-D-first-step}, we get
\begin{eqnarray*}
\left[D_N^{(\alpha,\beta)}\right]_{ik}
&=&
\frac{k+\alpha+\beta+1}{4}
\int_{s_i}
P_{k-1}^{(\alpha+1,\beta+1)}(x)
\omega_{\alpha+1,\beta}(x)\dd x
\\
&+&
\frac{k+\alpha+\beta+1}{4}
\int_{s_i}
P_{k-1}^{(\alpha+1,\beta+1)}(x)
\omega_{\alpha,\beta+1}(x)\dd x
\\
&=&
\frac{1}{2}
\left(
\left[A_N^{(\alpha,\beta)}\right]_{ik}
+
\left[B_N^{(\alpha,\beta)}\right]_{ik}
\right)
\left[C_N^{(\alpha,\beta)}\right]_{kk}.
\end{eqnarray*}
Then, the  factorization~\eqref{eq:endpoint-shifted-decomposition} follows.
\end{proof}

\begin{remark}\label{remarkimps}
The entries of the matrices
$A_N^{(\alpha,\beta)}$ and $B_N^{(\alpha,\beta)}$ can be written in closed
form in terms of incomplete beta functions evaluated at the endpoints of the
intervals. Let
\[
s_i=\left[a_i,b_i\right],
\quad
u_i^-=\frac{1+a_i}{2},
\quad
u_i^+=\frac{1+b_i}{2},
\]
and define
\[
\mathcal B_z(\lambda,\mu)
=
\int_0^z t^{\lambda-1}(1-t)^{\mu-1}\dd t,
\quad
\lambda,\mu>0.
\]
We shall use the classical expansion of the Jacobi polynomials
\[
P_n^{(a,b)}(x)
=
\sum_{m=0}^{n}
\binom{n+a}{n-m}
\binom{n+b}{m}
\left(\frac{x+1}{2}\right)^{n-m}
\left(\frac{x-1}{2}\right)^m ,
\]
where the binomial coefficients are understood in the generalized sense~\cite{Mastroianni:2008:IPS}. In
particular, taking $a=\alpha+1$, $b=\beta+1$ and $x=2u-1$, we obtain
\[
P_n^{(\alpha+1,\beta+1)}(2u-1)
=
\sum_{m=0}^{n}
(-1)^m
\binom{n+\alpha+1}{n-m}
\binom{n+\beta+1}{m}
u^{n-m}(1-u)^m .
\]
Hence, for $n=k-1$, we have
\begin{eqnarray*}
\left[A_N^{(\alpha,\beta)}\right]_{ik}
&=&
2^{\alpha+\beta+2}
\sum_{m=0}^{n}
(-1)^m
\binom{n+\alpha+1}{n-m}
\binom{n+\beta+1}{m}  \\
&&\times
\left[
\mathcal B_{u_i^+}
\left(n-m+\beta+1,m+\alpha+2\right)
-
\mathcal B_{u_i^-}
\left(n-m+\beta+1,m+\alpha+2\right)
\right],
\end{eqnarray*}
and
\begin{eqnarray*}
\left[B_N^{(\alpha,\beta)}\right]_{ik}
&=&
2^{\alpha+\beta+2}
\sum_{m=0}^{n}
(-1)^m
\binom{n+\alpha+1}{n-m}
\binom{n+\beta+1}{m}  \\
&&\times
\left[
\mathcal B_{u_i^+}
\left(n-m+\beta+2,m+\alpha+1\right)
-
\mathcal B_{u_i^-}
\left(n-m+\beta+2,m+\alpha+1\right)
\right].
\end{eqnarray*}
Therefore the decomposition~\eqref{eq:endpoint-shifted-decomposition} can be
evaluated using only endpoint quantities. In particular, for the four
Chebyshev choices
\[
(\alpha,\beta)\in
\left\{
\left(-\frac{1}{2},-\frac{1}{2}\right),
\left(\frac{1}{2},\frac{1}{2}\right),
\left(-\frac{1}{2},\frac{1}{2}\right),
\left(\frac{1}{2},-\frac{1}{2}\right)
\right\},
\]
all the parameters which occur in the incomplete beta functions are integers
or half-integers. Hence the above expressions reduce to elementary endpoint
formulae.
\end{remark}

\subsection{A Chebyshev realization with exact diagonalization}
\label{sec:chebyshev}

We present a configuration in which the reduced derivative matrix admits an
explicit spectral factorization. We assume that the derivative moments are
unweighted, namely
\[
\omega(x)=1.
\]
The trial space is represented in the first-kind Chebyshev basis
\[
\left\{
T_k(\cos\theta)=\cos(k\theta)\,:\, k=0,\ldots,N
\right\}.
\]
As a scalar normalization, we consider
\begin{equation}\label{eq:chebyshev-normalization}
\mathcal M_{\mathrm C}(f)
=
\frac{1}{\pi}
\int_{-1}^{1}
\frac{f(x)}{\sqrt{1-x^2}}\dd x,
\quad f\in C([-1,1]).
\end{equation}
Then, by the orthogonality of the first-kind Chebyshev polynomials, we have
\begin{equation}\label{supimp}
\mathcal M_{\mathrm C}\left(T_0\right)=1,
\quad
\mathcal M_{\mathrm C}\left(T_k\right)=0,
\quad k=1,\ldots,N.
\end{equation}
In particular,
\begin{equation*}
\mathcal M_{\mathrm C}(1)=1.
\end{equation*}
Let
\begin{equation}\label{eq:chebyshev-angles}
\theta_j
=
\frac{(2j+1)\pi}{2(N+1)},
\quad j=0,\ldots,N,
\end{equation}
and set
\begin{equation}\label{eq:chebyshev-nodes}
x_j=\cos\theta_j,
\quad j=0,\ldots,N.
\end{equation}
These are the $N+1$ zeros of $T_{N+1}$, ordered as
\[
x_0>x_1>\cdots>x_N.
\]
We define the $N$ cells
\begin{equation}\label{eq:chebyshev-cells}
s_i=\left[x_i,x_{i-1}\right],
\quad i=1,\ldots,N.
\end{equation}
Hence, up to the ordering of the nodes, Proposition~\ref{prop:elementary-intervals-admissible} shows that the family
\[
\mathcal S_{\mathrm C}
=
\left\{s_1,\ldots,s_N\right\}
\]
is $1$-admissible. Since $\mathcal M_{\mathrm C}(1)=1$, Theorem~\ref{thm:admissible-unisolvence}
implies that the set of degrees of freedom
\[
\left\{
\mathcal M_{\mathrm C},
\mathcal D_1^{1},\ldots,\mathcal D_N^{1}
\right\}
\]
is unisolvent on $\Pi_N$. In this case, the corresponding reduced derivative matrix
$D_N\in\RR^{N\times N}$ is defined by
\begin{equation*}
\left[D_N\right]_{ik}
=
\mathcal D_i^{1}\left(T_k\right)
=
\int_{s_i}T_k'(x)\dd x,
\quad i,k=1,\ldots,N.
\end{equation*}
The following theorem shows that, for the cells defined in~\eqref{eq:chebyshev-cells},
the reduced derivative matrix admits an exact factorization through the
discrete sine transform.

\begin{theorem}
\label{thm:chebyshev-factorization}
Let $S_N\in\RR^{N\times N}$ be defined by
\begin{equation*}
\left[S_N\right]_{ik}
=
\sin\left(\frac{ik\pi}{N+1}\right),
\quad i,k=1,\ldots,N,
\end{equation*}
and let
\begin{equation*}
\Sigma_N
=
\operatorname{diag}
\left(
2\sin\left(\frac{k\pi}{2(N+1)}\right)
\right)_{k=1}^{N}.
\end{equation*}
Then
\begin{equation}\label{eq:chebyshev-factorization}
D_N=S_N\Sigma_N.
\end{equation}
Consequently,
\begin{equation}\label{eq:chebyshev-gram-matrix}
D_N^\top D_N
=
\frac{N+1}{2}\Sigma_N^2.
\end{equation}
\end{theorem}

\begin{proof}
For any $i,k=1,\ldots,N$, by the fundamental theorem of calculus and by
the definition of the cells in~\eqref{eq:chebyshev-cells}, we have
\begin{equation*}
\left[D_N\right]_{ik}
=
\int_{s_i}T_k'(x)\dd x
=
T_k\left(x_{i-1}\right)-T_k\left(x_i\right).
\end{equation*}
Using~\eqref{eq:chebyshev-angles} and~\eqref{eq:chebyshev-nodes}, we obtain
\begin{eqnarray*}
\left[D_N\right]_{ik}
&=&
T_k\left(\cos \theta_{i-1}\right)-T_k\left(\cos \theta_i\right)
\\
&=&
\cos\left(k\theta_{i-1}\right)-\cos\left(k\theta_i\right)
\\
&=&
2
\sin\left(\frac{ik\pi}{N+1}\right)
\sin\left(\frac{k\pi}{2(N+1)}\right)
\\
&=&
\left[S_N\right]_{ik}
\left[\Sigma_N\right]_{kk}.
\end{eqnarray*}
Since this identity holds for any $i,k=1,\ldots,N$,
\eqref{eq:chebyshev-factorization} follows.

By the discrete sine orthogonality relation~\cite[Chapter 4]{Mason:2002:CP}
\begin{equation*}
S_N^\top S_N
=
\frac{N+1}{2}I_N,
\end{equation*}
we get
\begin{equation*}
D_N^\top D_N
=
\left(S_N\Sigma_N\right)^\top S_N\Sigma_N=
\Sigma_N S_N^\top S_N\Sigma_N=
\frac{N+1}{2}\Sigma_N^2.
\end{equation*}
\end{proof}

\begin{corollary}\label{cor:chebyshev-spectrum-condition-number}
The singular values of $D_N$ are
\begin{equation}\label{eq:chebyshev-singular-values}
\sigma_k\left(D_N\right)
=
\sqrt{2(N+1)}
\sin\left(\frac{k\pi}{2(N+1)}\right),
\quad k=1,\ldots,N.
\end{equation}
Consequently, the spectral condition number of $D_N$ is
\begin{equation}\label{eq:chebyshev-condition-number}
\kappa_2\left(D_N\right)
=
\cot\left(\frac{\pi}{2(N+1)}\right).
\end{equation}
\end{corollary}

\begin{proof}
By~\eqref{eq:chebyshev-gram-matrix}, the matrix $D_N^\top D_N$ is diagonal,
and its $k$-th diagonal entry is
\begin{equation*}
\frac{N+1}{2}
\left[
2\sin\left(\frac{k\pi}{2(N+1)}\right)
\right]^2
=
2(N+1)
\sin^2\left(\frac{k\pi}{2(N+1)}\right).
\end{equation*}
Thus the eigenvalues of $D_N^\top D_N$ are
\[
2(N+1)
\sin^2\left(\frac{k\pi}{2(N+1)}\right),
\quad k=1,\ldots,N.
\]
Taking square roots gives~\eqref{eq:chebyshev-singular-values}.

Since the sine function is strictly increasing on $(0,\pi/2)$, the smallest
and greatest singular values are respectively
\[
\sigma_{\min}\left(D_N\right)
=
\sqrt{2(N+1)}
\sin\left(\frac{\pi}{2(N+1)}\right)
\]
and
\[
\sigma_{\max}\left(D_N\right)
=
\sqrt{2(N+1)}
\sin\left(\frac{N\pi}{2(N+1)}\right).
\]
Therefore,
\begin{equation*}
\kappa_2\left(D_N\right)
=
\frac{\sigma_{\max}\left(D_N\right)}{\sigma_{\min}\left(D_N\right)}
=
\frac{
\sin\left(\dfrac{N\pi}{2(N+1)}\right)
}{
\sin\left(\dfrac{\pi}{2(N+1)}\right)
}.
\end{equation*}
Since
\[
\frac{N\pi}{2(N+1)}
=
\frac{\pi}{2}
-
\frac{\pi}{2(N+1)},
\]
we have
\[
\sin\left(\frac{N\pi}{2(N+1)}\right)
=
\cos\left(\frac{\pi}{2(N+1)}\right).
\]
Hence
\[
\kappa_2\left(D_N\right)
=
\frac{
\cos\left(\dfrac{\pi}{2(N+1)}\right)
}{
\sin\left(\dfrac{\pi}{2(N+1)}\right)
}
=
\cot\left(\frac{\pi}{2(N+1)}\right),
\]
which proves~\eqref{eq:chebyshev-condition-number}.
\end{proof}

The factorization above relies on the elementary identity
\[
\int_u^v p'(x)\dd x
=
p(v)-p(u),
\quad [u,v]\subset(-1,1).
\]
The next proposition shows that, within the Jacobi family, this property characterizes the unweighted case.

\begin{proposition}
\label{prop:telescoping-jacobi-weight}
Let $\alpha,\beta>-1$. Assume that there exists a constant $c\neq0$ such that
\begin{equation}\label{eq:telescoping-property}
\int_u^v
p'(x)\omega_{\alpha,\beta}(x)\dd x
=
c\left(p(v)-p(u)\right), \quad [u,v]\subset(-1,1),
\end{equation}
for any polynomial $p$. Then
\[
\alpha=\beta=0
\quad\text{and}\quad
c=1.
\]
\end{proposition}

\begin{proof}
Choosing $p(x)=x$ in~\eqref{eq:telescoping-property}, we obtain
\begin{equation}\label{eq:weight-interval-average}
\int_u^v\omega_{\alpha,\beta}(x)\dd x
=
c(v-u),
\quad [u,v]\subset(-1,1).
\end{equation}
Equivalently,
\[
\int_u^v\left(\omega_{\alpha,\beta}(x)-c\right)\dd x
=
0
\]
for any interval $[u,v]\subset(-1,1)$. Hence
\[
\omega_{\alpha,\beta}(x)=c
\quad\text{for almost every }x\in(-1,1).
\]
Since $\omega_{\alpha,\beta}$ is continuous on $(-1,1)$, it follows that
\[
\omega_{\alpha,\beta}(x)=c,
\quad x\in(-1,1).
\]
Then $\omega_{\alpha,\beta}'=0$. Therefore, by~\eqref{eq:jacobi-weight},
\[
\frac{\omega_{\alpha,\beta}'(x)}
{\omega_{\alpha,\beta}(x)}
=
-\frac{\alpha}{1-x}
+
\frac{\beta}{1+x}
=
\frac{-(\alpha+\beta)x+\beta-\alpha}{(1-x)(1+x)}
\]
vanishes identically on $(-1,1)$. Thus
\[
\alpha+\beta=0,
\quad
\beta-\alpha=0,
\]
and consequently
\[
\alpha=\beta=0.
\]
Since $\omega_{0,0}(x)=1$, equation~\eqref{eq:weight-interval-average} gives
$c=1$.
\end{proof}

\begin{remark}
In Proposition~\ref{prop:telescoping-jacobi-weight}, it is not necessary to
verify the identity for all polynomials. It is enough to test the single
linear polynomial $p(x)=x$. Indeed, this gives
\[
\int_u^v \omega_{\alpha,\beta}(x)\dd x=c(v-u)
\]
for every subinterval $[u,v]\subset(-1,1)$, and therefore
$\omega_{\alpha,\beta}$ must be constant almost everywhere. Hence, in the
Jacobi family, weighted derivative moments can reduce to endpoint differences
only for the constant weight $\omega_{0,0}=1$. This is the reason why the
Chebyshev factorizations obtained above are tied to the unweighted case.
\end{remark}

\subsection{Derivative histopolation and Chebyshev--Gauss interpolation}
\label{subsec:chebyshev-gauss-interpolation}

Let the nodes $x_j$ be defined by~\eqref{eq:chebyshev-nodes}. We denote by
\[
I_N^{\mathrm G}:C([-1,1])\to\Pi_N
\]
the Chebyshev--Gauss interpolation operator associated with
$x_0,\ldots,x_N$. Thus, for any $f\in C([-1,1])$, we have
\begin{equation}\label{eq:cheb-gauss-interpolation-data}
I_N^{\mathrm G}[f]\left(x_j\right)=f\left(x_j\right),
\quad j=0,\ldots,N.
\end{equation}
We also introduce the discrete Chebyshev--Gauss mean
\begin{equation*}
Q_N^{\mathrm C}(f)
=
\frac{1}{N+1}
\sum_{j=0}^{N}f\left(x_j\right).
\end{equation*}
Let
\[
K_N^{\mathrm C}:C^1([-1,1])\to\Pi_N
\]
be the derivative histopolation operator associated with the Chebyshev mean
$\mathcal M_{\mathrm C}$ defined in~\eqref{eq:chebyshev-normalization}. Thus
$K_N^{\mathrm C}[f]$ is the unique polynomial in $\Pi_N$ satisfying
\begin{equation}\label{eq:cheb-histopolant-mean}
\mathcal M_{\mathrm C}\left(K_N^{\mathrm C}[f]\right)
=
\mathcal M_{\mathrm C}(f)
\end{equation}
and
\begin{equation}\label{eq:cheb-histopolant-derivative-moments}
\int_{s_i}
\left(K_N^{\mathrm C}[f]\right)'(x)\dd x
=
\int_{s_i}f'(x)\dd x,
\quad i=1,\ldots,N.
\end{equation}
The following result shows that, in this Chebyshev configuration, derivative
histopolation differs from Chebyshev--Gauss interpolation only by an additive
constant.

\begin{proposition}\label{prop:cheb-gauss-interpolation-formula}
For any $f\in C^1([-1,1])$, we have
\begin{equation}\label{eq:cheb-histopolation-gauss-interpolation}
K_N^{\mathrm C}[f]
=
I_N^{\mathrm G}[f]
+
\left(
\mathcal M_{\mathrm C}(f)-Q_N^{\mathrm C}(f)
\right).
\end{equation}
\end{proposition}

\begin{proof}
By~\eqref{eq:cheb-histopolant-derivative-moments} and by the fundamental
theorem of calculus, for $i=1,\ldots,N$, we have
\begin{eqnarray*}
K_N^{\mathrm C}[f]\left(x_{i-1}\right)
-
K_N^{\mathrm C}[f]\left(x_i\right)
&=&
\int_{s_i}
\left(K_N^{\mathrm C}[f]\right)'(x)\dd x
\\
&=&
\int_{s_i}f'(x)\dd x
=
f\left(x_{i-1}\right)-f\left(x_i\right).
\end{eqnarray*}
Hence
\[
K_N^{\mathrm C}[f]\left(x_{i-1}\right)-f\left(x_{i-1}\right)
=
K_N^{\mathrm C}[f]\left(x_i\right)-f\left(x_i\right),
\quad i=1,\ldots,N.
\]
Therefore, there exists $c\in\RR$ such that
\[
K_N^{\mathrm C}[f]\left(x_j\right)-f\left(x_j\right)=c,
\quad j=0,\ldots,N.
\]
Since
$I_N^{\mathrm G}[f]\left(x_j\right)=f\left(x_j\right)$, it follows that
\[
K_N^{\mathrm C}[f]\left(x_j\right)-I_N^{\mathrm G}[f]\left(x_j\right)=c,
\quad j=0,\ldots,N.
\]
The polynomial
\[
K_N^{\mathrm C}[f]-I_N^{\mathrm G}[f]-c\in \Pi_N
\]
vanishes at the $N+1$ distinct nodes $x_0,\ldots,x_N$. Hence
\begin{equation}\label{eq:cheb-histopolant-interpolant-constant}
K_N^{\mathrm C}[f]
=
I_N^{\mathrm G}[f]+c.
\end{equation}
It remains to identify $c$. For any $q\in\Pi_N$, writing
\begin{equation}\label{eq:cheb-expansion-q}
q=\sum_{k=0}^{N}a_kT_k,
\end{equation}
we obtain, by~\eqref{supimp},
\begin{equation}\label{unascc}
    \mathcal M_{\mathrm C}(q)=a_0.
\end{equation}
On the other hand, the Chebyshev--Gauss nodes satisfy the discrete
orthogonality relation
\begin{equation}\label{eq:cheb-gauss-discrete-orthogonality}
\frac{1}{N+1}
\sum_{j=0}^{N}T_k\left(x_j\right)
=
\begin{cases}
1, & k=0,\\
0, & k=1,\ldots,N.
\end{cases}
\end{equation}
Indeed, writing
\[
h=\frac{\pi}{N+1},
\quad
\theta_j=\left(j+\frac{1}{2}\right)h,
\]
we have $x_j=\cos\theta_j$ and hence
\[
T_k\left(x_j\right)=\cos\left(k\theta_j\right).
\]
For $k=1,\ldots,N$, the elementary identity
\[
2\sin\left(\frac{kh}{2}\right)\cos\left(k\theta_j\right)
=
\sin(k(j+1)h)-\sin(kjh)
\]
gives, after summing over $j=0,\ldots,N$,
\begin{eqnarray*}
2\sin\left(\frac{kh}{2}\right)
\sum_{j=0}^{N}\cos\left(k\theta_j\right)
&=&
\sum_{j=0}^{N}
\left[
\sin(k(j+1)h)-\sin(kjh)
\right]
\\
&=&
\sin(k(N+1)h)-\sin(0)
\\
&=&
\sin(k\pi)
=
0.
\end{eqnarray*}
Since $\sin(kh/2)\neq0$ for $k=1,\ldots,N$, it follows that
\[
\sum_{j=0}^{N}\cos\left(k\theta_j\right)=0,
\quad k=1,\ldots,N.
\]
This proves~\eqref{eq:cheb-gauss-discrete-orthogonality}, the case $k=0$
being immediate from $T_0=1$.

Hence, by~\eqref{eq:cheb-expansion-q} and
\eqref{eq:cheb-gauss-discrete-orthogonality}, we have
\begin{eqnarray}
Q_N^{\mathrm C}(q)
&=&
\frac{1}{N+1}
\sum_{j=0}^{N}q\left(x_j\right)\notag
\\
&=&
\frac{1}{N+1}
\sum_{j=0}^{N}
\left(
\sum_{k=0}^{N}a_kT_k\left(x_j\right)
\right)\notag
\\
&=&
\sum_{k=0}^{N}a_k
\left(
\frac{1}{N+1}
\sum_{j=0}^{N}T_k\left(x_j\right)
\right)
=
a_0.\label{unascc1}
\end{eqnarray}
Thus, combining~\eqref{unascc} and~\eqref{unascc1}, we get
\begin{equation}\label{eq:cheb-mean-exactness}
\mathcal M_{\mathrm C}(q)=Q_N^{\mathrm C}(q),
\quad q\in\Pi_N.
\end{equation}
Applying $\mathcal M_{\mathrm C}$ to~\eqref{eq:cheb-histopolant-interpolant-constant}
and using~\eqref{eq:cheb-histopolant-mean} gives
\begin{equation}\label{consa}
    \mathcal M_{\mathrm C}(f)
=
\mathcal M_{\mathrm C}\left(I_N^{\mathrm G}[f]\right)+c.
\end{equation}
Since $I_N^{\mathrm G}[f]\in\Pi_N$, by~\eqref{eq:cheb-mean-exactness} and by
the interpolation conditions~\eqref{eq:cheb-gauss-interpolation-data}, we obtain
\begin{eqnarray*}
\mathcal M_{\mathrm C}\left(I_N^{\mathrm G}[f]\right)
&=&
Q_N^{\mathrm C}\left(I_N^{\mathrm G}[f]\right)
\\
&=&
\frac{1}{N+1}
\sum_{j=0}^{N}I_N^{\mathrm G}[f]\left(x_j\right)
\\
&=&
\frac{1}{N+1}
\sum_{j=0}^{N}f\left(x_j\right)
=
Q_N^{\mathrm C}(f).
\end{eqnarray*}
Therefore, by~\eqref{consa}, we have
\[
c=\mathcal M_{\mathrm C}(f)-Q_N^{\mathrm C}(f),
\]
which proves~\eqref{eq:cheb-histopolation-gauss-interpolation}.
\end{proof}

As a consequence of Proposition~\ref{prop:cheb-gauss-interpolation-formula},
the Chebyshev expansion of the derivative histopolant can be written in closed
form.

\begin{corollary}
\label{cor:cheb-gauss-coefficients}
Let $f\in C^1([-1,1])$. Then
\[
K_N^{\mathrm C}[f](x)
=
c_0^{(N)}
+
\sum_{k=1}^{N}c_k^{(N)}T_k(x),
\]
where
\begin{equation}\label{eq:cheb-histopolant-coefficients}
c_0^{(N)}
=
\mathcal M_{\mathrm C}(f),
\quad
c_k^{(N)}
=
\frac{2}{N+1}
\sum_{j=0}^{N}f\left(x_j\right)T_k\left(x_j\right),
\quad k=1,\ldots,N.
\end{equation}
\end{corollary}

\begin{proof}
Write
\[
K_N^{\mathrm C}[f](x)
=
c_0^{(N)}
+
\sum_{k=1}^{N}c_k^{(N)}T_k(x).
\]
By~\eqref{supimp} and by the normalization condition
\eqref{eq:cheb-histopolant-mean}, we obtain
\[
c_0^{(N)}
=
\mathcal M_{\mathrm C}\left(K_N^{\mathrm C}[f]\right)
=
\mathcal M_{\mathrm C}(f).
\]
It remains to identify the nonconstant coefficients. Let
\[
I_N^{\mathrm G}[f](x)
=
a_0^{(N)}
+
\sum_{\ell=1}^{N}a_\ell^{(N)}T_\ell(x).
\]
Evaluating this identity at the Chebyshev--Gauss nodes and multiplying by
$T_k\left(x_j\right)$, we get
\[
I_N^{\mathrm G}[f]\left(x_j\right)T_k\left(x_j\right)
=
a_0^{(N)}T_k\left(x_j\right)
+
\sum_{\ell=1}^{N}
a_\ell^{(N)}T_\ell\left(x_j\right)T_k\left(x_j\right).
\]
Summing over $j=0,\ldots,N$ and using the discrete orthogonality relations
\[
\sum_{j=0}^{N}T_k\left(x_j\right)=0,
\quad k=1,\ldots,N,
\]
and
\[
\frac{2}{N+1}
\sum_{j=0}^{N}T_\ell\left(x_j\right)T_k\left(x_j\right)
=
\begin{cases}
1, & \ell=k,\\
0, & \ell\neq k,
\end{cases}
\quad k,\ell=1,\ldots,N.
\]
we obtain
\[
a_k^{(N)}
=
\frac{2}{N+1}
\sum_{j=0}^{N}
I_N^{\mathrm G}[f]\left(x_j\right)T_k\left(x_j\right),
\quad k=1,\ldots,N.
\]
Since $I_N^{\mathrm G}[f]\left(x_j\right)=f\left(x_j\right)$ for $j=0,\ldots,N$, it follows that
\[
a_k^{(N)}
=
\frac{2}{N+1}
\sum_{j=0}^{N}
f\left(x_j\right)T_k\left(x_j\right),
\quad k=1,\ldots,N.
\]
Finally, Proposition~\ref{prop:cheb-gauss-interpolation-formula} gives
\[
K_N^{\mathrm C}[f]
=
I_N^{\mathrm G}[f]
+
\left(
\mathcal M_{\mathrm C}(f)-Q_N^{\mathrm C}(f)
\right).
\]
Thus $K_N^{\mathrm C}[f]$ and $I_N^{\mathrm G}[f]$ differ only by an additive
constant. Therefore their coefficients of $T_1,\ldots,T_N$ coincide, namely
\[
c_k^{(N)}=a_k^{(N)},
\quad k=1,\ldots,N.
\]
This proves~\eqref{eq:cheb-histopolant-coefficients}.
\end{proof}

\section{Weighted Jacobi primitives and Chebyshev grids}
\label{sec4}

In this section, we use weighted primitives to relate Jacobi derivative
histopolation to interpolation. Let $\alpha,\beta>-1$, and let
\[
1\ge x_0>x_1>\cdots>x_N\ge -1.
\]
Set
\[
s_i=\left[x_i,x_{i-1}\right],
\quad i=1,\ldots,N.
\]
Since the intervals $s_i$ are consecutive cells, Proposition~\ref{prop:elementary-intervals-admissible} implies
that the family
\[
\mathcal S=\left\{s_1,\ldots,s_N\right\}
\]
is $\omega_{\alpha,\beta}$-admissible. Moreover,
\[
\mathcal M_{\alpha,\beta}(1)
=
\int_{-1}^{1}\omega_{\alpha,\beta}(x)\dd x
>0.
\]
Therefore, by Theorem~\ref{thm:admissible-unisolvence}, the corresponding
weighted derivative histopolation problem is unisolvent on $\Pi_N$. 
For any
$u\in C^1([-1,1])$, define 
\begin{equation}\label{eq:jacobi-weighted-primitive}
W_{\alpha,\beta}[u](x)
=
\int_{x_0}^{x}u'(t)\omega_{\alpha,\beta}(t)\dd t,
\quad x\in[-1,1].
\end{equation}
Let $K_N^{(\alpha,\beta)}[f]\in\Pi_N$ be the weighted derivative
histopolant of $f\in C^1([-1,1])$, namely the unique polynomial satisfying
\begin{equation}\label{eq:jacobi-weighted-histopolant-mean}
\mathcal M_{\alpha,\beta}
\left(K_N^{(\alpha,\beta)}[f]\right)
=
\mathcal M_{\alpha,\beta}(f),
\end{equation}
where $\mathcal M_{\alpha,\beta}$ is defined in~\eqref{eq:jacobi-mean}, and
\begin{equation}\label{eq:jacobi-weighted-histopolant-moments}
\int_{s_i}
\left(K_N^{(\alpha,\beta)}[f]\right)'(x)
\omega_{\alpha,\beta}(x)\dd x
=
\int_{s_i}
f'(x)\omega_{\alpha,\beta}(x)\dd x,
\quad i=1,\ldots,N.
\end{equation}

\begin{proposition}\label{prop:weighted-primitive-interpolation}
For any $f\in C^1([-1,1])$, we have
\begin{equation}\label{eq:weighted-primitive-nodal-agreement}
W_{\alpha,\beta}
\left[
K_N^{(\alpha,\beta)}[f]
\right]\left(x_j\right)
=
W_{\alpha,\beta}[f]\left(x_j\right),
\quad j=0,\ldots,N.
\end{equation}
Hence, if $I_N$ denotes the interpolation operator at the nodes
$x_0,\ldots,x_N$, then
\begin{equation}\label{eq:weighted-primitive-interpolation}
I_N
\left[
W_{\alpha,\beta}
\left[
K_N^{(\alpha,\beta)}[f]
\right]
\right]
=
I_N
\left[
W_{\alpha,\beta}[f]
\right].
\end{equation}
\end{proposition}

\begin{proof}
By definition~\eqref{eq:jacobi-weighted-primitive}, we have
\[
W_{\alpha,\beta}
\left[
K_N^{(\alpha,\beta)}[f]
\right]\left(x_0\right)
=
W_{\alpha,\beta}[f]\left(x_0\right)
=
0.
\]
Moreover, by~\eqref{eq:jacobi-weighted-histopolant-moments}, for
$i=1,\ldots,N$ we have
\begin{eqnarray*}
&&
W_{\alpha,\beta}
\left[
K_N^{(\alpha,\beta)}[f]
\right]\left(x_{i-1}\right)
-
W_{\alpha,\beta}
\left[
K_N^{(\alpha,\beta)}[f]
\right]\left(x_i\right)
\\
&=&
\int_{x_i}^{x_{i-1}}
\left(K_N^{(\alpha,\beta)}[f]\right)'(x)
\omega_{\alpha,\beta}(x)\dd x
\\
&=&
\int_{x_i}^{x_{i-1}}
f'(x)\omega_{\alpha,\beta}(x)\dd x
\\
&=&
W_{\alpha,\beta}[f]\left(x_{i-1}\right)
-
W_{\alpha,\beta}[f]\left(x_i\right).
\end{eqnarray*}
Since the two weighted primitives coincide at $x_0$, the previous identity gives
successively their agreement at $x_1,\ldots,x_N$. This proves
\eqref{eq:weighted-primitive-nodal-agreement}. The identity for the interpolants follows at once.
\end{proof}

 Let
\begin{equation}\label{eq:chebyshev-lobatto-angles}
\theta_j=\frac{j\pi}{N},
\quad
x_j=\cos\left(\theta_j\right),
\quad j=0,\ldots,N.
\end{equation}
Then
\[
1=x_0>x_1>\cdots>x_N=-1.
\]
Set
\[
s_i=\left[x_i,x_{i-1}\right],
\quad i=1,\ldots,N,
\]
and define the reduced Chebyshev derivative matrix by
\begin{equation}\label{eq:chebyshev-lobatto-derivative-matrix}
\left[\widetilde D_N\right]_{ik}
=
\int_{s_i}T_k'(x)\dd x,
\quad i,k=1,\ldots,N.
\end{equation}

\begin{proposition}\label{prop:chebyshev-lobatto-factorization}
Let
\begin{equation}\label{eq:chebyshev-lobatto-sine-matrix}
\left[\widetilde S_N\right]_{ik}
=
\sin\left(\frac{\left(2i-1\right)k\pi}{2N}\right),
\quad i,k=1,\ldots,N,
\end{equation}
and
\begin{equation}\label{eq:chebyshev-lobatto-diagonal-factor}
\widetilde\Sigma_N
=
\operatorname{diag}
\left(
2\sin\left(\frac{k\pi}{2N}\right)
\right)_{k=1}^{N}.
\end{equation}
Then
\begin{equation}\label{eq:chebyshev-lobatto-factorization}
\widetilde D_N
=
\widetilde S_N\widetilde\Sigma_N.
\end{equation}
Moreover,
\begin{equation}\label{eq:chebyshev-lobatto-sine-orthogonality}
\widetilde S_N^\top\widetilde S_N
=
\operatorname{diag}
\left(
\frac N2,\ldots,\frac N2,N
\right).
\end{equation}
\end{proposition}

\begin{proof}
By the fundamental theorem of calculus,
\[
\left[\widetilde D_N\right]_{ik}
=
T_k\left(x_{i-1}\right)-T_k\left(x_i\right).
\]
Using~\eqref{eq:chebyshev-lobatto-angles}, we obtain
\begin{eqnarray*}
\left[\widetilde D_N\right]_{ik}
&=&
\cos\left(\frac{\left(i-1\right)k\pi}{N}\right)
-
\cos\left(\frac{ik\pi}{N}\right)
\\
&=&
2
\sin\left(\frac{\left(2i-1\right)k\pi}{2N}\right)
\sin\left(\frac{k\pi}{2N}\right)\\
&=&\left[\widetilde S_N\right]_{ik} \left[\widetilde\Sigma_N \right]_{kk}.
\end{eqnarray*}
This gives~\eqref{eq:chebyshev-lobatto-factorization}. Moreover, the standard
discrete sine orthogonality relation~\cite[Chapter~4]{Mason:2002:CP} gives
\begin{eqnarray*}
\left[\widetilde S_N^\top\widetilde S_N\right]_{km}
&=&
\sum_{i=1}^{N}
\sin\left(\frac{\left(2i-1\right)k\pi}{2N}\right)
\sin\left(\frac{\left(2i-1\right)m\pi}{2N}\right)
\\
&=&
\begin{cases}
0, & k\neq m,\\[0.4ex]
\dfrac N2, & k=m,\quad k=1,\ldots,N-1,\\[1ex]
N, & k=m=N.
\end{cases}
\end{eqnarray*}
Hence~\eqref{eq:chebyshev-lobatto-sine-orthogonality} follows.
\end{proof}

\begin{remark}
The two Chebyshev grids considered above correspond to the two standard sine
structures associated with the interior and half-shifted angular meshes. Such
sine and cosine based matrix spaces, and their role in optimal matrix algebra
approximations, are discussed in~\cite{Benedetto:2000:OMM}.
\end{remark}

\begin{corollary}\label{cor:chebyshev-lobatto-spectrum}
The matrix $\widetilde D_N^\top\widetilde D_N$ is diagonal. Its diagonal
entries are
\begin{equation}\label{eq:chebyshev-lobatto-gram-entries}
\left[\widetilde D_N^\top\widetilde D_N\right]_{kk}
=
2N\sin^2\left(\frac{k\pi}{2N}\right),
\quad k=1,\ldots,N-1,
\end{equation}
and
\begin{equation}\label{eq:chebyshev-lobatto-last-entry}
\left[\widetilde D_N^\top\widetilde D_N\right]_{NN}
=
4N.
\end{equation}
Consequently,
\[
\sigma_k\left(\widetilde D_N\right)
=
\sqrt{2N}
\sin\left(\frac{k\pi}{2N}\right),
\quad k=1,\ldots,N-1, \quad \sigma_N\left(\widetilde D_N\right)
=
2\sqrt N.
\]
\end{corollary}

\begin{proof}
By Proposition~\ref{prop:chebyshev-lobatto-factorization}, we have
\[
\widetilde D_N^\top\widetilde D_N
=\left(\widetilde S_N
\widetilde\Sigma_N \right)^{\top} \left(\widetilde S_N
\widetilde\Sigma_N\right)=
\widetilde\Sigma_N
\widetilde S_N^\top
\widetilde S_N
\widetilde\Sigma_N.
\]
Together with~\eqref{eq:chebyshev-lobatto-sine-orthogonality}, this proves
that $\widetilde D_N^\top\widetilde D_N$ is diagonal. For
$k=1,\ldots,N-1$, the $k$-th diagonal entry is
\[
\frac N2
\left[
2\sin\left(\frac{k\pi}{2N}\right)
\right]^2
=
2N\sin^2\left(\frac{k\pi}{2N}\right).
\]
For $k=N$, it is
\[
N
\left[
2\sin\left(\frac{\pi}{2}\right)
\right]^2
=
4N.
\]
Taking square roots gives the stated singular values.
\end{proof}

\begin{remark}\label{rem:chebyshev-lobatto-condition-number} 
For $N\ge 2$, the singular values in Corollary~\ref{cor:chebyshev-lobatto-spectrum} give
directly the spectral condition number of $\widetilde D_N$. Since
\[
\sin\left(\frac{k\pi}{2N}\right)
\]
is increasing for $k=1,\ldots,N-1$, the smallest singular value is
\[
\sigma_{\min}\left(\widetilde D_N\right)
=
\sqrt{2N}
\sin\left(\frac{\pi}{2N}\right),
\]
whereas the greatest one is
\[
\sigma_{\max}\left(\widetilde D_N\right)
=
2\sqrt N.
\]
Therefore,
\begin{equation*}
\kappa_2\left(\widetilde D_N\right)
=
\frac{\sigma_{\max}\left(\widetilde D_N\right)}
{\sigma_{\min}\left(\widetilde D_N\right)}
=
\frac{\sqrt{2}}
{\sin\left(\frac{\pi}{2N}\right)}.
\end{equation*}
In particular,
\[
\kappa_2\left(\widetilde D_N\right)
\sim
\frac{2\sqrt{2}}{\pi}N,
\quad N\to\infty.
\]
Thus the Chebyshev--Lobatto reduced derivative matrix has a condition number
which grows linearly with the polynomial degree.
\end{remark}

The behaviour of the two closed form expressions for
$\kappa_2\left(D_N\right)$ and
$\kappa_2\left(\widetilde D_N\right)$ is shown in
Figure~\ref{fig:cheb-conditioning}.
\begin{figure}[h]
\centering
\includegraphics[width=0.49\textwidth]{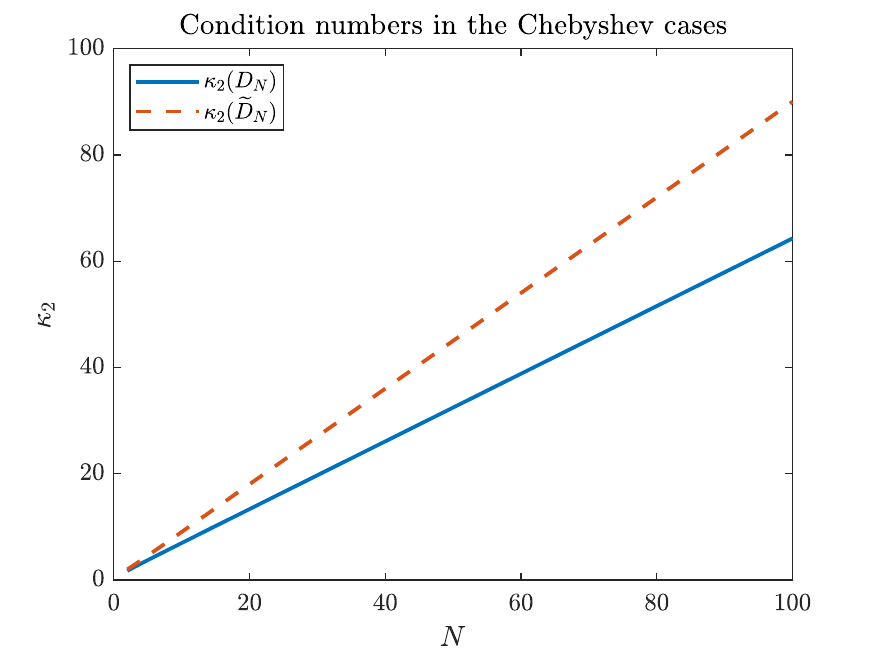}
\caption{Spectral condition numbers of the reduced derivative
matrices in the Chebyshev cases.}
\label{fig:cheb-conditioning}
\end{figure}

\begin{remark}
The grid introduced in~\eqref{eq:chebyshev-lobatto-angles} is the
Chebyshev--Gauss--Lobatto grid. This grid is standard in Chebyshev spectral
collocation methods~\cite{Serra:2000:CTF,Capizzano:2003:AOP}. In that
setting, the unknown function is usually represented by its values at the
nodes, and derivatives are approximated by differentiation matrices. In the
present construction the same nodes are used in a different way. They define
the cells on which the derivative moments are imposed. For this particular
mesh, the resulting derivative moment matrix has an exact discrete sine
transform structure, as proved in
Proposition~\ref{prop:chebyshev-lobatto-factorization}. This shows that the
Chebyshev--Gauss--Lobatto grid also has a natural role in derivative
histopolation, where it leads to the same type of trigonometric structure
that appears in classical Chebyshev spectral discretizations.
\end{remark}

We next show that the diagonal structure of the reduced Gram matrix is not
determined only by the weight and by the interval family. Indeed, if the
polynomial basis is not prescribed, a diagonal form can always be obtained by
a suitable choice of basis.

\begin{theorem}\label{thm:decomp}
Let $\omega\in L^1(-1,1)$ satisfy~\eqref{eq:positive-weight}, and let
\[
\mathcal S=\left\{s_1,\ldots,s_N\right\}
\]
be an $\omega$-admissible family of nondegenerate intervals. Let $V$ be an
algebraic complement of $\Pi_0$ in $\Pi_N$, namely
\[
\Pi_N=\Pi_0\oplus V.
\]
Then there exists a basis $\left\{\phi_1,\ldots,\phi_N\right\}$ of $V$ such that the reduced derivative matrix
\[
D_N=\left[
\int_{s_i}\phi_k'(x)\omega(x)\dd x
\right]_{i,k=1}^{N}
\]
satisfies
\[
D_N^\top D_N=I_N.
\]
\end{theorem}
\begin{proof}
For $p,q\in V$, set
\[
\langle p,q\rangle_{\mathcal S,\omega}
=
\sum_{i=1}^{N}
\left(
\int_{s_i}p'(x)\omega(x)\dd x
\right)
\left(
\int_{s_i}q'(x)\omega(x)\dd x
\right).
\]
This is a symmetric bilinear form on $V$. We show that it is positive
definite. Let $p\in V$ and assume that
\[
\langle p,p\rangle_{\mathcal S,\omega}=\sum_{i=1}^{N}
\left(
\int_{s_i}p'(x)\omega(x)\dd x
\right)^2=0.
\]
Then each term in the sum is nonnegative, and therefore
\[
\mathcal D_i^\omega(p)=\int_{s_i}p'(x)\omega(x)\dd x=0,
\quad i=1,\ldots,N.
\]
Since $\mathcal S$ is $\omega$-admissible, Proposition~\ref{prop:admissible-derivative-kernel}
implies that $p$ is constant. 
Hence $p\in V\cap\Pi_0$. Since
\[
\Pi_N=\Pi_0\oplus V,
\]
we have $V\cap\Pi_0=\{0\}$. Thus $p=0$. Consequently,
$\langle\cdot,\cdot\rangle_{\mathcal S,\omega}$ is an inner product on $V$.

We may therefore choose an orthonormal basis
\[
\left\{\phi_1,\ldots,\phi_N\right\}
\]
of $V$ with respect to this inner product. For this basis, and for any $j,k=1,\ldots,N$, we have
\begin{eqnarray*}
\left[D_N^\top D_N\right]_{jk}
&=&
\sum_{i=1}^{N}
\left(
\int_{s_i}\phi_j'(x)\omega(x)\dd x
\right)
\left(
\int_{s_i}\phi_k'(x)\omega(x)\dd x
\right)\\
&=&
\left\langle\phi_j,\phi_k\right\rangle_{\mathcal S,\omega}=
\begin{cases}
1, & j=k,\\
0, & j\neq k.
\end{cases}
\end{eqnarray*}
Therefore
\[
D_N^\top D_N=I_N.
\]
\end{proof}

Let $\alpha,\beta,\gamma,\delta>-1$, and let
\[
\mathcal S=\left\{s_1,\ldots,s_N\right\}
\]
be an $\omega_{\gamma,\delta}$-admissible family of intervals. We denote by
$\mathcal M_{\alpha,\beta}$ the linear functional defined in
\eqref{eq:jacobi-mean}, namely
\[
\mathcal M_{\alpha,\beta}(f)
=
\int_{-1}^{1}f(x)\omega_{\alpha,\beta}(x)\dd x.
\]
We define
\[
V_{\alpha,\beta}
=
\operatorname{span}
\left\{
P_1^{(\alpha,\beta)},\ldots,P_N^{(\alpha,\beta)}
\right\}.
\]
By the orthogonality of the Jacobi polynomials with respect to
$\omega_{\alpha,\beta}$, we have
\[
\mathcal M_{\alpha,\beta}
\left(P_k^{(\alpha,\beta)}\right)
=
\int_{-1}^{1}
P_k^{(\alpha,\beta)}(x)\omega_{\alpha,\beta}(x)\dd x
=
0,
\quad k=1,\ldots,N.
\]
Hence, by linearity,
\[
\mathcal M_{\alpha,\beta}(p)=0,
\quad p\in V_{\alpha,\beta}.
\]
Since
\[
\mathcal M_{\alpha,\beta}
\left(P_0^{(\alpha,\beta)}\right)
=
\mathcal M_{\alpha,\beta}(1)
>0,
\]
the restriction of $\mathcal M_{\alpha,\beta}$ to $\Pi_N$ is a nonzero
linear functional. Therefore its kernel in $\Pi_N$ has dimension $N$. Since
$V_{\alpha,\beta}$ has also dimension $N$, we obtain
\begin{equation}\label{Valphabeta}
V_{\alpha,\beta}
=
\left\{
p\in\Pi_N:
\mathcal M_{\alpha,\beta}(p)=0
\right\}.
\end{equation}
In particular,
\[
V_{\alpha,\beta}\cap\Pi_0=\{0\},
\]
and consequently
\[
\Pi_N=\Pi_0\oplus V_{\alpha,\beta}.
\]

\begin{corollary}\label{cor:mixed-jacobi-diagonal-gram}
There exists a basis
\[
\Phi=\left\{\phi_1,\ldots,\phi_N\right\}
\]
of $V_{\alpha,\beta}$ such that the matrix representation of the map
$\mathcal J_{\mathcal S}^{(\alpha,\beta;\gamma,\delta)}$, defined
in~\eqref{Jacobi_Map}, with respect to the basis
\[
\left\{1,\phi_1,\ldots,\phi_N\right\}
\]
has the block form
\begin{equation}\label{blockform}
\begin{bmatrix}
h_0^{(\alpha,\beta)} & \boldsymbol 0^\top\\[2mm]
\boldsymbol 0 & D_{\Phi,N}^{(\alpha,\beta;\gamma,\delta)}
\end{bmatrix},
\end{equation}
where
\[
\left(
D_{\Phi,N}^{(\alpha,\beta;\gamma,\delta)}
\right)^\top
D_{\Phi,N}^{(\alpha,\beta;\gamma,\delta)}
=
I_N.
\]
\end{corollary}

\begin{proof}
We first observe that $V_{\alpha,\beta}$ is an algebraic complement of
$\Pi_0$ in $\Pi_N$. Indeed,
\[
\mathcal M_{\alpha,\beta}(1)
=
\int_{-1}^{1}\omega_{\alpha,\beta}(x)\dd x
=
h_0^{(\alpha,\beta)}
>0.
\]
Hence, for any $p\in\Pi_N$, we can write
\[
p=
c+
\left(p-c\right),
\quad
c=
\frac{\mathcal M_{\alpha,\beta}(p)}
     {\mathcal M_{\alpha,\beta}(1)}.
\]
Then
\[
\mathcal M_{\alpha,\beta}(p-c)
=
\mathcal M_{\alpha,\beta}(p)
-
c\mathcal M_{\alpha,\beta}(1)
=
0,
\]
and therefore
\[
p-c\in V_{\alpha,\beta}.
\]
It remains to show that
\[
\Pi_0\cap V_{\alpha,\beta}=\{0\}.
\]
Let
\[
p\in \Pi_0\cap V_{\alpha,\beta}.
\]
Then $p=c$ for some $c\in\RR$. Since $p\in V_{\alpha,\beta}$, by~\eqref{Valphabeta}, we have
\[
0=
\mathcal M_{\alpha,\beta}(p)
=
c\mathcal M_{\alpha,\beta}(1).
\]
Since $\mathcal M_{\alpha,\beta}(1)>0$, it follows that $c=0$. Consequently,
\[
\Pi_N=\Pi_0\oplus V_{\alpha,\beta}.
\]
Then, by Theorem~\ref{thm:decomp} there exists a basis
\[
\Phi=\left\{\phi_1,\ldots,\phi_N\right\}
\]
of $V_{\alpha,\beta}$ such that the corresponding reduced derivative matrix
satisfies
\[
\left(
D_{\Phi,N}^{(\alpha,\beta;\gamma,\delta)}
\right)^\top
D_{\Phi,N}^{(\alpha,\beta;\gamma,\delta)}
=
I_N,
\]
where
\[
D_{\Phi,N}^{(\alpha,\beta;\gamma,\delta)}
=
\left[
\int_{s_i}\phi_k'(x)\omega_{\gamma,\delta}(x)\dd x
\right]_{i,k=1}^{N}.
\]

It remains to identify the matrix representation of
$\mathcal J_{\mathcal S}^{(\alpha,\beta;\gamma,\delta)}$ in the basis
\[
\left\{1,\phi_1,\ldots,\phi_N\right\}.
\]
For the first basis function, namely the constant polynomial $1$, we have
\[
\mathcal M_{\alpha,\beta}(1)=h_0^{(\alpha,\beta)}
\]
and
\[
\mathcal D_i^{\omega_{\gamma,\delta}}(1)=0,
\quad i=1,\ldots,N.
\]
Hence the first column of the matrix is
\[
\left[
h_0^{(\alpha,\beta)},0,\ldots,0
\right]^\top .
\]
For the remaining basis functions, since
\[
\phi_k\in V_{\alpha,\beta},
\quad k=1,\ldots,N,
\]
by~\eqref{Valphabeta}, we have
\[
\mathcal M_{\alpha,\beta}\left(\phi_k\right)=0,
\quad k=1,\ldots,N.
\]
Hence the first row vanishes outside the first entry. Therefore the matrix has
the block form
\[
\begin{bmatrix}
h_0^{(\alpha,\beta)} & \boldsymbol 0^\top\\[2mm]
\boldsymbol 0 & D_{\Phi,N}^{(\alpha,\beta;\gamma,\delta)}
\end{bmatrix}.
\]
This proves~\eqref{blockform} and completes the proof.
\end{proof}

\begin{example}\label{ex:chebyshev-weight-diagonal-gram}
We present an explicit Chebyshev case in which a basis of
$V_{-\frac{1}{2},-\frac{1}{2}}$ with diagonal reduced Gram matrix can be written in
closed form. Let
\[
\alpha=\beta=\gamma=\delta=-\frac{1}{2}.
\]
Then
\[
\omega_{\gamma,\delta}(x)
=
\omega_{-\frac{1}{2},-\frac{1}{2}}(x)
=
\frac{1}{\sqrt{1-x^2}},
\quad x\in(-1,1).
\]
Let
\[
x_i=\cos\left(\frac{i\pi}{N}\right),
\quad i=0,\ldots,N,
\]
and consider the intervals
\[
s_i=\left[x_i,x_{i-1}\right],
\quad i=1,\ldots,N.
\]
By
Proposition~\ref{prop:elementary-intervals-admissible}, the set
\[
\mathcal S=\left\{s_1,\ldots,s_N\right\}
\]
is $\omega_{-\frac{1}{2},-\frac{1}{2}}$-admissible. For
$k=1,\ldots,N$, let $\phi_k\in\Pi_N$ be the unique polynomial satisfying
\begin{equation}\label{eq:cheb-weight-basis}
\phi_k'(x)=T_{k-1}(x),
\quad
\mathcal M_{-\frac{1}{2},-\frac{1}{2}}\left(\phi_k\right)=0.
\end{equation}
The second condition fixes the additive constant. In particular, by the
characterization of $V_{\alpha,\beta}$ in~\eqref{Valphabeta}, we have
\[
\phi_k\in V_{-\frac{1}{2},-\frac{1}{2}},
\quad k=1,\ldots,N.
\]
Moreover, if
\[
\sum_{k=1}^{N}c_k\phi_k=0,
\]
then, after differentiating, we obtain
\[
\sum_{k=1}^{N}c_kT_{k-1}=0.
\]
Since
\[
T_0,T_1,\ldots,T_{N-1}
\]
are linearly independent, it follows that
\[
c_1=\cdots=c_N=0.
\]
Thus
\[
\Phi=\left\{\phi_1,\ldots,\phi_N\right\}
\]
is a basis of $V_{-\frac{1}{2},-\frac{1}{2}}$. We consider the derivative matrix associated with this basis, defined by
\[
\left[
D_{\Phi,N}^{\left(-\frac{1}{2},-\frac{1}{2};-\frac{1}{2},-\frac{1}{2}\right)}
\right]_{ik}
=
\int_{s_i}
\phi_k'(x)\frac{\dd x}{\sqrt{1-x^2}},
\quad i,k=1,\ldots,N.
\]
Using~\eqref{eq:cheb-weight-basis}, this becomes
\begin{equation}\label{eq:cheb-weight-matrix-entry}
\left[
D_{\Phi,N}^{\left(-\frac{1}{2},-\frac{1}{2};-\frac{1}{2},-\frac{1}{2}\right)}
\right]_{ik}
=
\int_{x_i}^{x_{i-1}}
T_{k-1}(x)\frac{\dd x}{\sqrt{1-x^2}}.
\end{equation}
We compute this integral explicitly. Set
\[
x=\cos\theta,
\quad
\dd x=-\sin\theta\dd\theta,
\]
and
\[
\theta_i=\frac{i\pi}{N},
\quad i=0,\ldots,N.
\]
Since
\[
T_{k-1}(\cos\theta)
=
\cos\left((k-1)\theta\right)
\]
and
\[
\frac{\dd x}{\sqrt{1-x^2}}
=
-\dd\theta,
\]
formula~\eqref{eq:cheb-weight-matrix-entry} gives
\[
\left[
D_{\Phi,N}^{\left(-\frac{1}{2},-\frac{1}{2};-\frac{1}{2},-\frac{1}{2}\right)}
\right]_{ik}
=
\int_{\theta_{i-1}}^{\theta_i}
\cos\left((k-1)\theta\right)\dd\theta.
\]
Therefore, for $k=1$, we have
\[
\left[
D_{\Phi,N}^{\left(-\frac{1}{2},-\frac{1}{2};-\frac{1}{2},-\frac{1}{2}\right)}
\right]_{i1}
=
\frac{\pi}{N},
\quad i=1,\ldots,N,
\]
while, for $k=2,\ldots,N$,
\begin{eqnarray*}
\left[
D_{\Phi,N}^{\left(-\frac{1}{2},-\frac{1}{2};-\frac{1}{2},-\frac{1}{2}\right)}
\right]_{ik}
&=&
\frac{
\sin\left((k-1)\theta_i\right)
-
\sin\left((k-1)\theta_{i-1}\right)
}{k-1}
\\[1mm]
&=&
\frac{2}{k-1}
\sin\left(\frac{(k-1)\pi}{2N}\right)
\cos\left(\frac{(k-1)(2i-1)\pi}{2N}\right).
\end{eqnarray*}
We now use the discrete cosine orthogonality on the midpoint grid. Let
\begin{equation}\label{thetastar}
\theta_i^\ast
=
\frac{(2i-1)\pi}{2N},
\quad i=1,\ldots,N.
\end{equation}
For $m,n=0,\ldots,N-1$, we have
\[
\sum_{i=1}^{N}
\cos\left(m\theta_i^\ast\right)
\cos\left(n\theta_i^\ast\right)
=
\begin{cases}
N, & m=n=0,\\[1mm]
\frac{N}{2}, & m=n\geq 1,\\[1mm]
0, & m\neq n.
\end{cases}
\]
Then, for $k,\ell=2,\ldots,N$, with $k\neq \ell$, using~\eqref{thetastar},
we obtain
\begin{eqnarray*}
&&\left[
\left(
D_{\Phi,N}^{\left(-\frac{1}{2},-\frac{1}{2};-\frac{1}{2},-\frac{1}{2}\right)}
\right)^\top
D_{\Phi,N}^{\left(-\frac{1}{2},-\frac{1}{2};-\frac{1}{2},-\frac{1}{2}\right)}
\right]_{k\ell} \\
&=&
\sum_{i=1}^{N}
\left[
D_{\Phi,N}^{\left(-\frac{1}{2},-\frac{1}{2};-\frac{1}{2},-\frac{1}{2}\right)}
\right]_{ik}
\left[
D_{\Phi,N}^{\left(-\frac{1}{2},-\frac{1}{2};-\frac{1}{2},-\frac{1}{2}\right)}
\right]_{i\ell}
\\
&=&
\frac{4}{(k-1)(\ell-1)}
\sin\left(\frac{(k-1)\pi}{2N}\right)
\sin\left(\frac{(\ell-1)\pi}{2N}\right)
\sum_{i=1}^{N}
\cos\left((k-1)\theta_i^\ast\right)
\cos\left((\ell-1)\theta_i^\ast\right)
\\
&=&
0.
\end{eqnarray*}
Moreover, if $k=1$ and $\ell=2,\ldots,N$, then
\begin{eqnarray*}
&&\left[
\left(
D_{\Phi,N}^{\left(-\frac{1}{2},-\frac{1}{2};-\frac{1}{2},-\frac{1}{2}\right)}
\right)^\top
D_{\Phi,N}^{\left(-\frac{1}{2},-\frac{1}{2};-\frac{1}{2},-\frac{1}{2}\right)}
\right]_{1\ell}\\
&=&
\sum_{i=1}^{N}
\left[
D_{\Phi,N}^{\left(-\frac{1}{2},-\frac{1}{2};-\frac{1}{2},-\frac{1}{2}\right)}
\right]_{i1}
\left[
D_{\Phi,N}^{\left(-\frac{1}{2},-\frac{1}{2};-\frac{1}{2},-\frac{1}{2}\right)}
\right]_{i\ell}
\\
&=&
\frac{\pi}{N}
\frac{2}{\ell-1}
\sin\left(\frac{(\ell-1)\pi}{2N}\right)
\sum_{i=1}^{N}
\cos\left((\ell-1)\theta_i^\ast\right)=
0.
\end{eqnarray*}
Hence the columns of $D_{\Phi,N}^{\left(-\frac{1}{2},-\frac{1}{2};-\frac{1}{2},-\frac{1}{2}\right)}$ 
are pairwise orthogonal. Therefore
\[
\left(
D_{\Phi,N}^{\left(-\frac{1}{2},-\frac{1}{2};-\frac{1}{2},-\frac{1}{2}\right)}
\right)^\top
D_{\Phi,N}^{\left(-\frac{1}{2},-\frac{1}{2};-\frac{1}{2},-\frac{1}{2}\right)}
=
\operatorname{diag}
\left(
\lambda_1,\ldots,\lambda_N
\right),
\]
where
\[
\lambda_1
=
\frac{\pi^2}{N}, \quad \lambda_k
=
\frac{2N}{(k-1)^2}
\sin^2\left(\frac{(k-1)\pi}{2N}\right), \quad k=2,\ldots,N.
\]
If $N=1$, the matrix has only one singular value and therefore
\[
\kappa_2\left(
D_{\Phi,N}^{\left(-\frac{1}{2},-\frac{1}{2};-\frac{1}{2},-\frac{1}{2}\right)}
\right)=1.
\]
Assume now that $N\geq2$. The singular values of $D_{\Phi,N}^{\left(-\frac{1}{2},-\frac{1}{2};-\frac{1}{2},-\frac{1}{2}\right)}$ are
\[
\sigma_k=\sqrt{\lambda_k},
\quad k=1,\ldots,N.
\]
We identify the greatest and the smallest diagonal entries. For
$k=2,\ldots,N$, the estimate $\sin x\leq x$ gives
\[
\lambda_k
\leq
\frac{\pi^2}{2N}
<
\frac{\pi^2}{N}
=
\lambda_1.
\]
Moreover, since the function
\[
t\mapsto \frac{\sin t}{t}
\]
is decreasing on $(0,\pi/2]$, the sequence
\[
\lambda_2,\ldots,\lambda_N
\]
is decreasing. Hence the smallest diagonal entry is
\[
\lambda_N
=
\frac{2N}{(N-1)^2}
\sin^2\left(\frac{(N-1)\pi}{2N}\right)
=
\frac{2N}{(N-1)^2}
\cos^2\left(\frac{\pi}{2N}\right).
\]
It follows that
\[
\kappa_2\left(
D_{\Phi,N}^{\left(-\frac{1}{2},-\frac{1}{2};-\frac{1}{2},-\frac{1}{2}\right)}
\right)
=
\sqrt{\frac{\lambda_1}{\lambda_N}}
=
\frac{\pi (N-1)}
{\sqrt{2}N\cos\left(\frac{\pi}{2N}\right)}.
\]
In particular,
\[
\lim_{N\to\infty}
\kappa_2\left(
D_{\Phi,N}^{\left(-\frac{1}{2},-\frac{1}{2};-\frac{1}{2},-\frac{1}{2}\right)}
\right)
=
\frac{\pi}{\sqrt{2}}.
\]
Thus, in this explicit primitive basis, the reduced derivative matrix is
uniformly well conditioned with respect to $N$. Finally, by setting
\[
\widehat\phi_k
=
\frac{\phi_k}{\sqrt{\lambda_k}},
\quad k=1,\ldots,N,
\]
and
\[
\widehat\Phi
=
\left\{\widehat\phi_1,\ldots,\widehat\phi_N\right\},
\]
then
\[
\left(
D_{\widehat\Phi,N}^{\left(-\frac{1}{2},-\frac{1}{2};-\frac{1}{2},-\frac{1}{2}\right)}
\right)^\top
D_{\widehat\Phi,N}^{\left(-\frac{1}{2},-\frac{1}{2};-\frac{1}{2},-\frac{1}{2}\right)}
=
I_N.
\]
Therefore, after this normalization, the above construction gives the
orthonormal basis of $V_{-\frac{1}{2},-\frac{1}{2}}$, with respect to the derivative
moment inner product, whose existence was proved in
Corollary~\ref{cor:mixed-jacobi-diagonal-gram}.
\end{example}

\begin{remark}
The previous example shows that, for the Chebyshev weight of the first kind,
the natural primitive basis already gives a diagonal reduced Gram matrix. We
now point out that the same diagonal structure can be obtained for all four
Chebyshev weights, provided that the primitive basis is chosen in a suitable
way.

We consider the Chebyshev--Lobatto grid
\begin{equation*}
x_i=\cos\left(\frac{i\pi}{N}\right),
\quad i=0,\ldots,N,
\end{equation*}
and the interval family
\begin{equation*}
s_i=\left[x_i,x_{i-1}\right],
\quad i=1,\ldots,N.
\end{equation*}
Equivalently, after the change of variables $x=\cos\theta$, we set
\begin{equation*}
\theta_i=\frac{i\pi}{N},
\quad
\theta_i^\ast=\frac{(2i-1)\pi}{2N},
\quad i=1,\ldots,N.
\end{equation*}
For $m\geq0$, define the vector $\boldsymbol b_m\in\RR^N$ by
\begin{equation*}
\left[\boldsymbol b_m\right]_i
=
\int_{\theta_{i-1}}^{\theta_i}\cos(m\theta)\dd\theta,
\quad i=1,\ldots,N.
\end{equation*}
Then
\begin{equation*}
\left[\boldsymbol b_0\right]_i
=
\frac{\pi}{N},
\quad i=1,\ldots,N,
\end{equation*}
and, for $m\geq1$,
\begin{equation*}
\left[\boldsymbol b_m\right]_i
=
\frac{2}{m}
\sin\left(\frac{m\pi}{2N}\right)
\cos\left(m\theta_i^\ast\right),
\quad i=1,\ldots,N.
\end{equation*}
By the discrete cosine orthogonality on the midpoint grid, the vectors
\begin{equation*}
\boldsymbol b_0,\boldsymbol b_1,\ldots,\boldsymbol b_{N-1}
\end{equation*}
are pairwise orthogonal. Moreover,
\begin{equation*}
\boldsymbol b_N=\boldsymbol 0.
\end{equation*}
Indeed, for $i=1,\ldots,N$, we have
\begin{eqnarray*}
\left[\boldsymbol b_N\right]_i
&=&
\int_{\theta_{i-1}}^{\theta_i}\cos(N\theta)\dd\theta
\\
&=&
\frac{1}{N}
\left(
\sin(N\theta_i)-\sin(N\theta_{i-1})
\right)
\\
&=&
\frac{1}{N}
\left(
\sin(i\pi)-\sin((i-1)\pi)
\right)
=
0.
\end{eqnarray*}
Consequently,
\begin{equation*}
\boldsymbol b_m^\top \boldsymbol b_n=0,
\quad m\neq n,
\quad 0\leq m,n\leq N-1.
\end{equation*}
Their squared norms are
\begin{equation*}
\left\|\boldsymbol b_0\right\|_2^2
=
\frac{\pi^2}{N},
\end{equation*}
and, for $m=1,\ldots,N-1$,
\begin{equation*}
\left\|\boldsymbol b_m\right\|_2^2
=
\frac{2N}{m^2}
\sin^2\left(\frac{m\pi}{2N}\right).
\end{equation*}
Thus, if the columns of the reduced derivative matrix are 
\begin{equation*}
\boldsymbol b_0,\boldsymbol b_1,\ldots,\boldsymbol b_{N-1},
\end{equation*}
then the corresponding reduced Gram matrix is diagonal. We now describe, for the four Chebyshev weights, explicit primitive bases
which have this property. We denote by
\begin{equation*}
T_m,\quad U_m,\quad V_m,\quad W_m
\end{equation*}
the Chebyshev polynomials of the first, second, third and fourth kind,
respectively. The associated Chebyshev weights are
\begin{equation*}
\omega_T(x)=\frac{1}{\sqrt{1-x^2}},
\quad
\omega_U(x)=\sqrt{1-x^2},
\end{equation*}
and
\begin{equation*}
\omega_V(x)=\sqrt{\frac{1+x}{1-x}},
\quad
\omega_W(x)=\sqrt{\frac{1-x}{1+x}}.
\end{equation*}
In each case, the additive constant in the primitive basis is fixed by the
corresponding weighted mean condition. More precisely, if $\omega_Q$ is one
of the above weights, we impose
\begin{equation*}
\mathcal M_Q(\psi)
=
\int_{-1}^{1}\psi(x)\omega_Q(x)\dd x
=
0.
\end{equation*}
The four choices of derivative basis are summarized in the following table.
Here $j=0,\ldots,N-1$, and the polynomial $\psi_{j+1}^{Q}$ is always
defined by prescribing its derivative and by imposing
\begin{equation*}
\mathcal M_Q\left(\psi_{j+1}^{Q}\right)=0.
\end{equation*}

\begin{center}
\footnotesize
\setlength{\tabcolsep}{2pt}
\renewcommand{\arraystretch}{1.45}
\begin{tabular}{c|c|p{0.58\textwidth}}
\toprule
$Q$ &
$\omega_Q$ &
Derivative of the primitive basis
\\
\midrule
$T$ &
$\displaystyle \frac{1}{\sqrt{1-x^2}}$
&
$\displaystyle
\left(\psi_{j+1}^{T}\right)'(x)=T_j(x)$
\\[2mm]
\midrule
$U$ &
$\displaystyle \sqrt{1-x^2}$
&
If $j\equiv N-2 \pmod 2$, then
\begin{equation*}
\left(\psi_{j+1}^{U}\right)'(x)
=
2\sum_{q=0}^{\frac{N-2-j}{2}} U_{j+2q}(x).
\end{equation*}
\\[-2mm]
\cline{3-3}
$U$ &
$\displaystyle \sqrt{1-x^2}$
&
If $j\equiv N-1 \pmod 2$, then
\begin{equation*}
\left(\psi_{j+1}^{U}\right)'(x)
=
2\sum_{\substack{q\geq0\\ j+2q\leq N-3}} U_{j+2q}(x)
+
\frac{N+1}{N}U_{N-1}(x).
\end{equation*}
\\[-2mm]
\midrule
$V$ &
$\displaystyle \sqrt{\frac{1+x}{1-x}}$
&
$\displaystyle
\left(\psi_{j+1}^{V}\right)'(x)
=
\sum_{m=j}^{N-1}(-1)^{m-j}V_m(x)$
\\[3mm]
\midrule
$W$ &
$\displaystyle \sqrt{\frac{1-x}{1+x}}$
&
$\displaystyle
\left(\psi_{j+1}^{W}\right)'(x)
=
\sum_{m=j}^{N-1}W_m(x)$
\\[3mm]
\bottomrule
\end{tabular}
\end{center}
In all four cases, the corresponding derivative moment column is
$\boldsymbol b_j$. Therefore, if $D_{\Psi^Q,N}$ denotes the reduced
derivative matrix associated with the primitive basis $\Psi^Q$, then
\begin{equation*}
D_{\Psi^Q,N}^\top D_{\Psi^Q,N}
=
\operatorname{diag}
\left(
\lambda_1,\ldots,\lambda_N
\right),
\quad
Q\in\{T,U,V,W\},
\end{equation*}
where
\begin{equation*}
\lambda_1=\frac{\pi^2}{N},
\quad
\lambda_k
=
\frac{2N}{(k-1)^2}
\sin^2\left(\frac{(k-1)\pi}{2N}\right),
\quad k=2,\ldots,N.
\end{equation*}
Thus all four Chebyshev weights admit explicit primitive bases which are
orthogonal with respect to the derivative moment inner product. 
\end{remark}

\begin{remark}
A similar construction can be obtained in the Legendre case. Let
$\omega=1$, and consider again the Chebyshev--Lobatto grid. For $k=1,\ldots,N$, let $\psi_k\in\Pi_N$ be defined by
\begin{equation*}
\psi_k'(x)=U_{k-1}(x),
\quad
\int_{-1}^{1}\psi_k(x)\dd x=0.
\end{equation*}
Then, using the change of variables $x=\cos\theta$, the entries of the corresponding reduced derivative matrix are
\begin{eqnarray*}
\left[D_N^{\rm L}\right]_{ik}
&=&
\int_{x_i}^{x_{i-1}}U_{k-1}(x)\dd x
\\
&=&
\int_{\theta_{i-1}}^{\theta_i}\sin(k\theta)\dd\theta
\\
&=&
\frac{2}{k}
\sin\left(\frac{k\pi}{2N}\right)
\sin\left(k\theta_i^\ast\right),
\end{eqnarray*}
where $\theta_i^\ast$ is defined in~\eqref{thetastar}. Hence, the columns of $D_N^{\rm L}$ are
proportional to the discrete sine vectors on the midpoint grid. By the discrete sine orthogonality relation, the reduced Gram matrix
\[
\left(D_N^{\rm L}\right)^\top D_N^{\rm L}
\]
is diagonal.
\end{remark}

\begin{remark}\label{rem-choice-grid}
The diagonal identities proved in this section depend on the simultaneous
choice of the weight, the basis, and the grid. In the Chebyshev--Lobatto
case this compatibility is explicit. After the change of variables
$x=\cos\theta$, the intervals
\[
s_i=\left[x_i,x_{i-1}\right],
\quad i=1,\ldots,N,
\]
correspond to intervals of equal length in the variable $\theta$. Moreover,
the entries of the derivative moment matrix are expressed in terms of sine
and cosine functions on this discrete angular grid. Hence the diagonal form
of the reduced Gram matrix follows from the discrete trigonometric
orthogonality used above. For general Jacobi parameters, this argument does not apply if the
Chebyshev--Lobatto grid is kept fixed. Indeed, if
\[
\left[D_N\right]_{ik}
=
\int_{s_i}\psi_k'(x)\omega_{\alpha,\beta}(x)\dd x,
\quad i,k=1,\ldots,N,
\]
then
\[
\left[D_N^\top D_N\right]_{k\ell}
=
\sum_{i=1}^{N}
\left(
\int_{s_i}\psi_k'(x)\omega_{\alpha,\beta}(x)\dd x
\right)
\left(
\int_{s_i}\psi_\ell'(x)\omega_{\alpha,\beta}(x)\dd x
\right).
\]
The orthogonality of the Jacobi polynomials gives integral identities on the
whole interval $[-1,1]$ with respect to the weight
$\omega_{\alpha,\beta}$. It does not give, in general, the vanishing of the
finite sum of products of cell moments above. Thus the discrete orthogonality used in the Chebyshev cases is not available,
in general, for Jacobi weights on the same Chebyshev--Lobatto cells. For a fixed pair $(\alpha,\beta)$, a different possibility is to use the
Jacobi--Gauss--Lobatto grid associated with the shifted parameters $(\alpha+1, \beta+1)$. Its interior
nodes are the zeros of $P_{N-1}^{(\alpha+1,\beta+1)}$. Together with a primitive basis satisfying
\[
\widetilde\psi_k'(x)
=
P_{k-1}^{(\alpha+1,\beta+1)}(x),
\quad k=1,\ldots,N,
\]
this choice makes the grid, the weight, and the derivative basis depend on
the same Jacobi parameters. However, the interval family then changes when
$(\alpha,\beta)$ changes. The choice made here is different. We keep the Chebyshev--Lobatto grid fixed.
Therefore the cells $s_i$ do not depend on $(\alpha,\beta)$, and different
Jacobi weights are compared on the same interval family. With this fixed
geometry, the reduced Gram matrix in the original primitive basis is not
diagonal in general. This is exactly the situation treated in the numerical
experiments. Finally, the diagonalization result of
Corollary~\ref{cor:mixed-jacobi-diagonal-gram} remains valid for every
admissible interval family. If $\widetilde D_N$ is any nonsingular reduced
derivative matrix and
\[
\widetilde D_N^\top\widetilde D_N=R^\top R
\]
is its Cholesky factorization, then the change of basis
\[
\psi_k
=
\sum_{j=1}^{N}
\left[R^{-1}\right]_{jk}\widetilde\psi_j,
\quad k=1,\ldots,N,
\]
gives the new reduced matrix
\[
D_N=\widetilde D_N R^{-1}.
\]
Consequently,
\[
D_N^\top D_N
=
R^{-\top}
\widetilde D_N^\top\widetilde D_N
R^{-1}
=
I_N.
\]
Thus the identity Gram matrix is always obtained after a suitable change of
basis, while the fixed Chebyshev--Lobatto grid keeps the interval family
independent of the Jacobi parameters.
\end{remark}

\section{Numerical tests}
\label{sec5}
In this section, we present a series of numerical experiments aimed at
evaluating the performance of the proposed method. We consider the three test
functions
\[
f_1(x)=e^x\cos(3x),
\quad
f_2(x)=\frac{1}{1+25x^2},
\quad
f_3(x)=\sin(4\pi x),
\quad x\in[-1,1],
\]
and we set the Jacobi weight $\omega_{\alpha,\beta}$ with
\[
(\alpha,\beta)=\left(\frac{1}{2},\frac{1}{2}\right).
\]
For any fixed $N\in\mathbb N$, two sets of nodes are used. The first one is the
equispaced grid
\[
X_N^{\mathrm{eq}}
=
\left\{
x_0^{\mathrm{eq}},\ldots,x_N^{\mathrm{eq}}
\right\},
\quad
x_j^{\mathrm{eq}}
=
-1+\frac{2j}{N},
\quad
j=0,\ldots,N,
\]
and the second one is the Chebyshev--Lobatto grid
\[
X_N^{\mathrm{CL}}
=
\left\{
x_0^{\mathrm{CL}},\ldots,x_N^{\mathrm{CL}}
\right\},
\quad
x_j^{\mathrm{CL}}
=
-\cos\left(\frac{j\pi}{N}\right),
\quad
j=0,\ldots,N.
\]
On each grid we consider the four families of intervals
\begin{itemize}
    \item[-] $\mathcal S_L=\left\{\left[x_0,x_i\right]\,:\, i=1,\ldots,N\right\};$

    \item[-] $\mathcal S_R=\left\{\left[x_{i-1},x_N\right]\,:\, i=1,\ldots,N\right\};$

    \item[-] $\mathcal E
    =
    \left\{
    \left[x_{i-1},x_i\right]\,:\, i=1,\ldots,N
    \right\};$

    \item[-] $\mathcal S_{\rm ov}
    =
    \left\{
    s_i^{\rm ov}\,:\, i=1,\ldots,N
    \right\}$, where
    \[
    s_i^{\rm ov}
    =
    \left[x_{i-1},x_{i+1}\right],
    \quad
    i=1,\ldots,N-1,
    \quad
    s_N^{\rm ov}
    =
    \left[x_0,x_1\right].
    \]
\end{itemize}
For a fixed function $f$ and a fixed family
$\mathcal S=\left\{s_1,\ldots,s_N\right\}$, we compute the polynomial
$K_N[f]\in\Pi_N$ such that
\[
\mathcal M_{\alpha,\beta}\left(K_N[f]\right)
=
\mathcal M_{\alpha,\beta}(f)
\]
and
\[
\int_{s_i}
\left(K_N[f]\right)'(x)\omega_{\alpha,\beta}(x)\dd x
=
\int_{s_i}
f'(x)\omega_{\alpha,\beta}(x)\dd x,
\quad i=1,\ldots,N.
\]
In the implementation, the polynomial $K_N[f]$ is represented in the Jacobi
basis
\[
K_N[f](x)
=
\sum_{k=0}^{N}c_k P_k^{(\alpha,\beta)}(x).
\]
By the orthogonality of the Jacobi polynomials, the normalization equation
gives the first row of the linear system. The remaining $N$ rows are obtained
from the derivative moment matrix
\[
D_N^{(\alpha,\beta)}
=
\frac{1}{2}
\left(
A_N^{(\alpha,\beta)}+B_N^{(\alpha,\beta)}
\right)
C_N^{(\alpha,\beta)}.
\]
The matrices $A_N^{(\alpha,\beta)}$ and $B_N^{(\alpha,\beta)}$ are assembled
by using the endpoint formulae given in Remark~\ref{remarkimps}. Thus the
matrix of the linear system is computed from explicit endpoint quantities,
without applying numerical quadrature to its entries. The only quantities evaluated by quadrature are the data depending on the
function $f$, namely
\[
\mathcal M_{\alpha,\beta}(f)
=
\int_{-1}^{1}f(x)\omega_{\alpha,\beta}(x)\dd x
\]
and
\[
\int_{s_i}
f'(x)\omega_{\alpha,\beta}(x)\dd x,
\quad i=1,\ldots,N.
\]
These integrals are computed by Gaussian quadrature. The error is 
measured by the maximum norm
\[
E_N
=
\left\|f-K_N[f]\right\|_\infty .
\]
In Figs.~\ref{fig:error-f1-uniform}, \ref{fig:error-f2-uniform}
and~\ref{fig:error-f3-uniform}, we report the error behaviour on the
equispaced and Chebyshev--Lobatto grids for the different families of
intervals, with
\[
N=5,10,15,20,25,30.
\]

\begin{figure}[ht]
\centering
\includegraphics[width=0.49\textwidth]{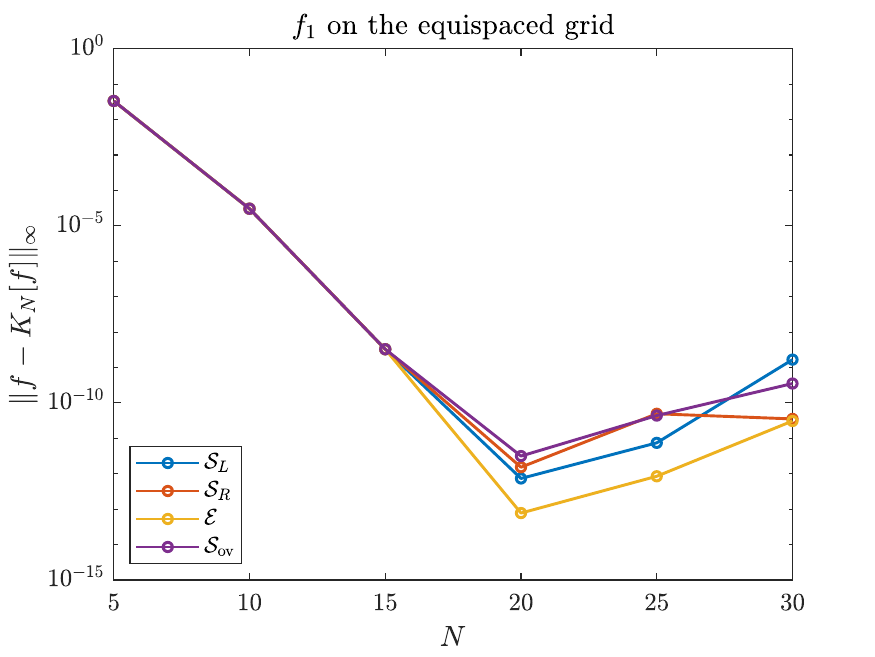}
\includegraphics[width=0.49\textwidth]{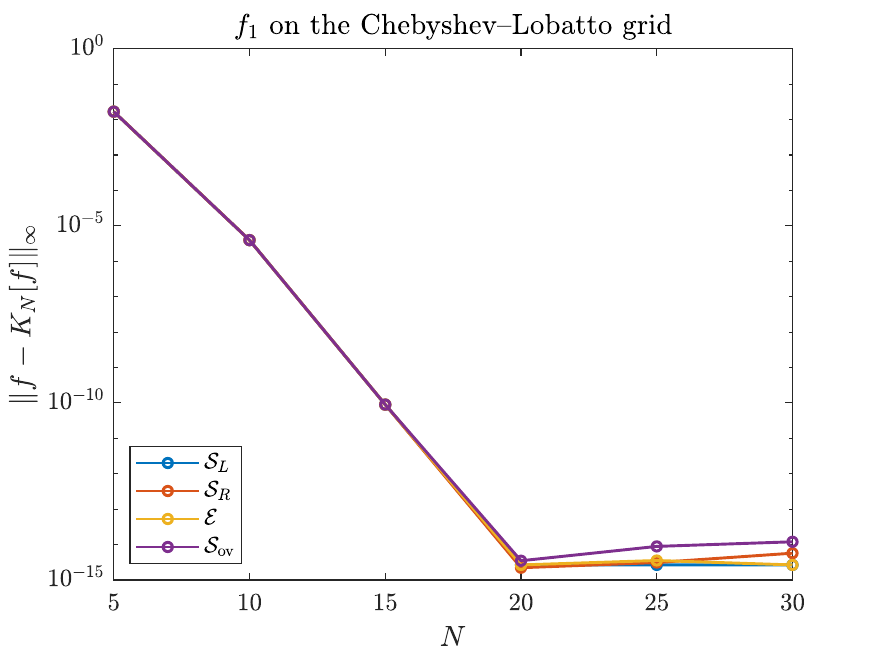}
\caption{Maximum error for $f_1$ on the equispaced grid (left) and on the
Chebyshev--Lobatto grid (right).}
\label{fig:error-f1-uniform}
\end{figure}

\begin{figure}[ht]
\centering
\includegraphics[width=0.49\textwidth]{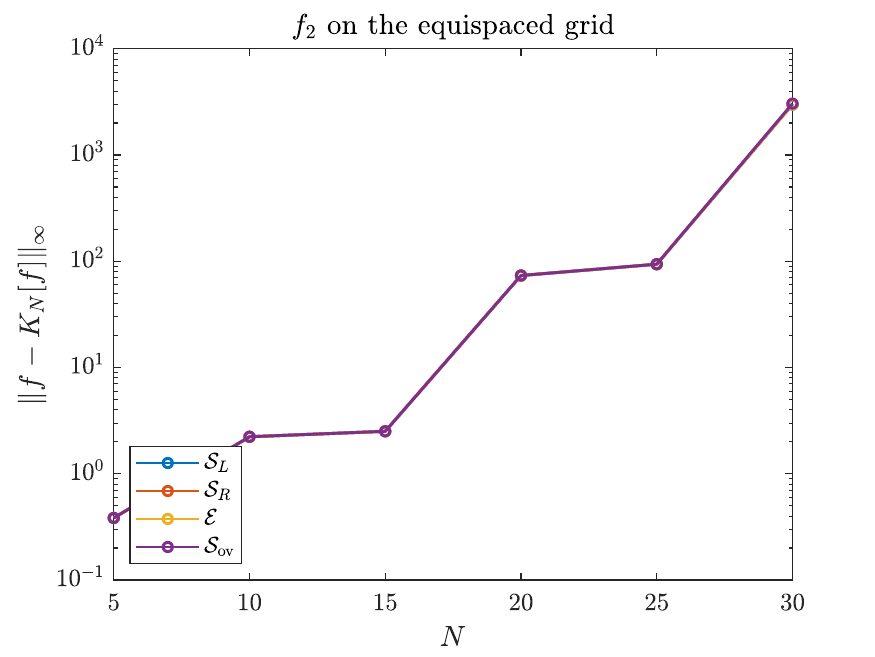}
\includegraphics[width=0.49\textwidth]{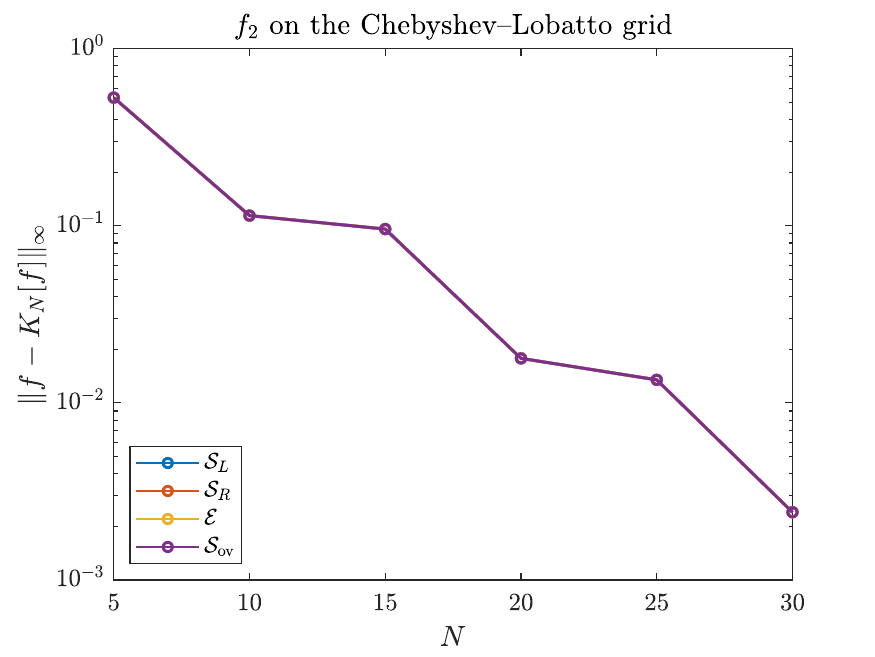}
\caption{Maximum error for $f_2$ on the equispaced grid (left) and on the
Chebyshev--Lobatto grid (right).}
\label{fig:error-f2-uniform}
\end{figure}

\begin{figure}[ht]
\centering
\includegraphics[width=0.49\textwidth]{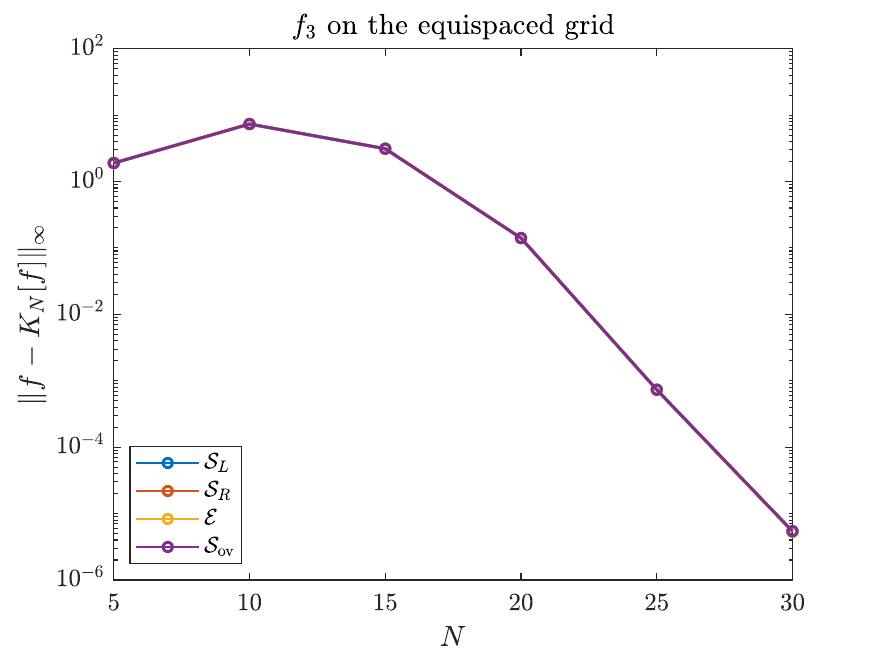}
\includegraphics[width=0.49\textwidth]{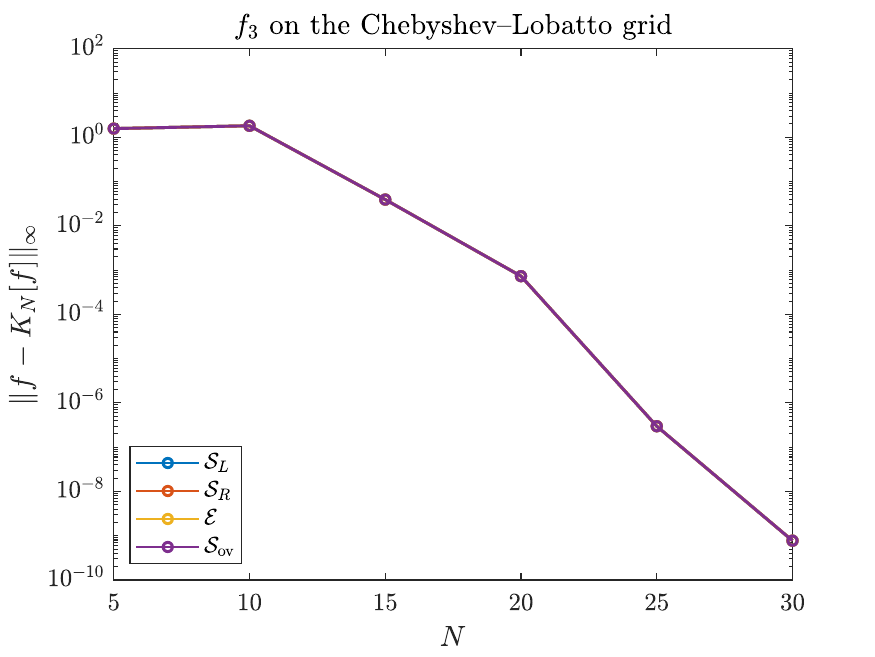}
\caption{Maximum error for $f_3$ on the equispaced grid (left) and on the
Chebyshev--Lobatto grid (right).}
\label{fig:error-f3-uniform}
\end{figure}

The results show that the use of equispaced nodes leads, in general, to a
less accurate approximation than the one obtained with Chebyshev--Lobatto
nodes, as expected. This behaviour is particularly evident for the function
$f_2$, which is the classical Runge-type test function. In this case, for all
the interval families considered, the error associated with the equispaced
grid increases as $N$ grows.

In Fig.~\ref{fig:new}, we show the test functions $f_i$, $i=1,2,3$, and the
corresponding approximations obtained with the Chebyshev--Lobatto
grid for $N=10$ and $N=20$, using the elementary family of intervals
\begin{equation}\label{Ep}
    \mathcal E
=
\left\{
\left[x_{i-1},x_i\right]\,:\, i=1,\ldots,N
\right\}.
\end{equation}
This comparison is included to illustrate the qualitative behaviour of the
approximants for the three test functions.

\begin{figure}[ht]
\centering
\includegraphics[width=0.32\textwidth]{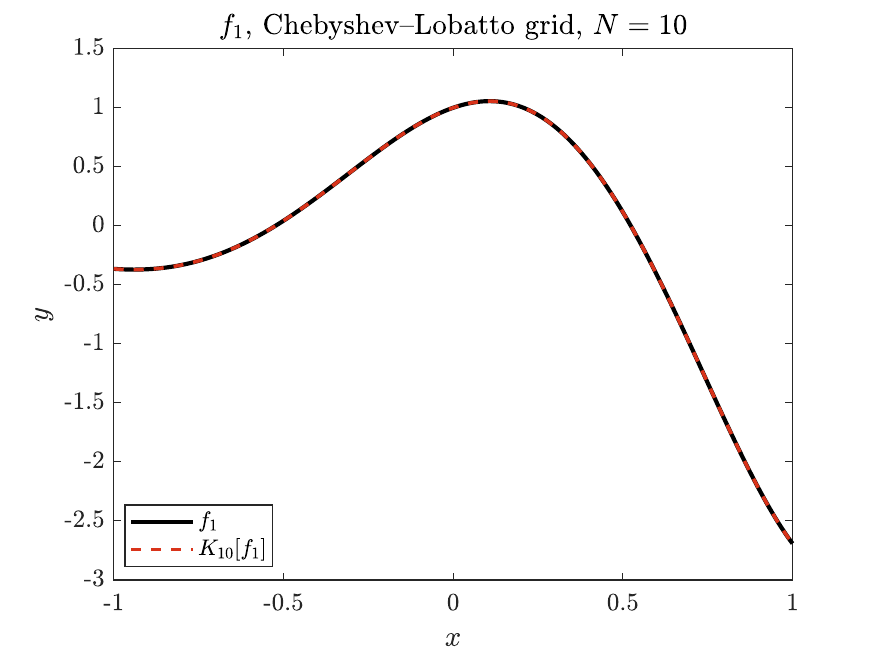}
\includegraphics[width=0.32\textwidth]{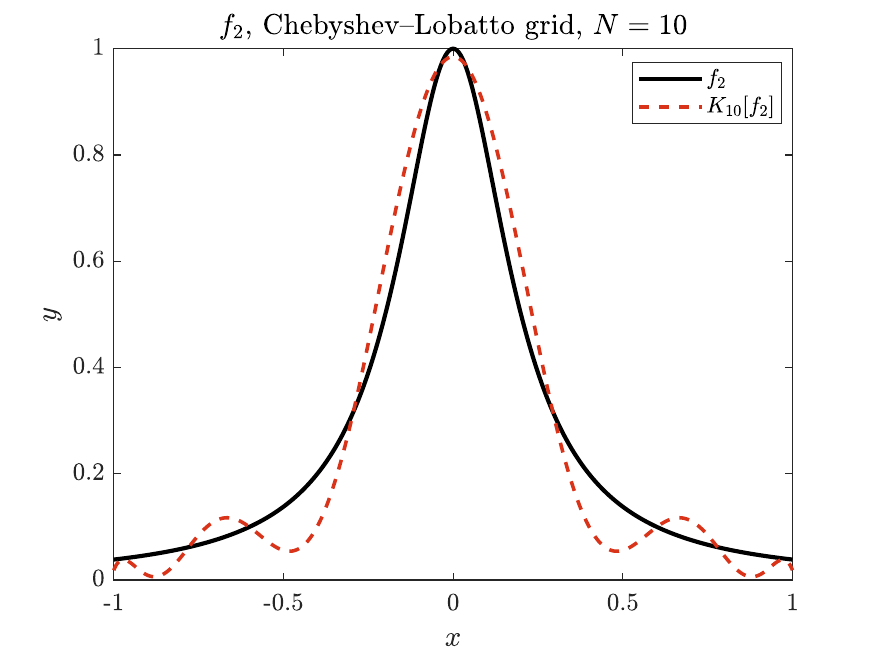}
\includegraphics[width=0.32\textwidth]{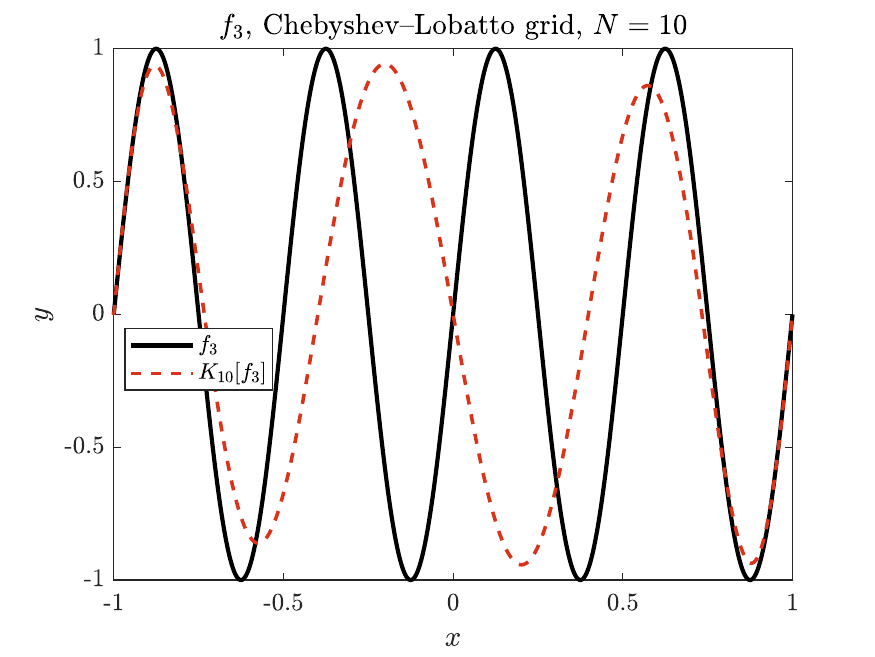}
\includegraphics[width=0.32\textwidth]{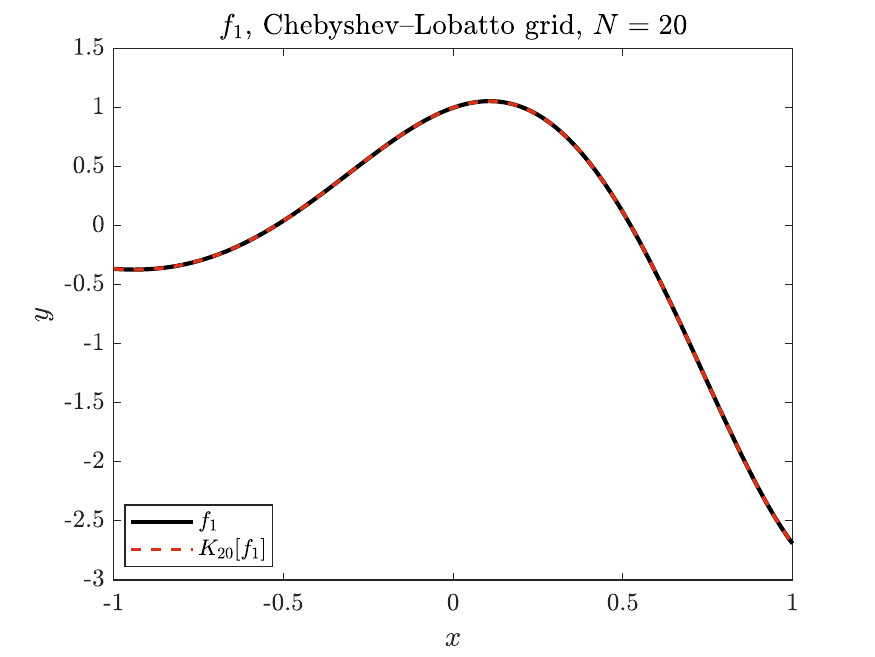}
\includegraphics[width=0.32\textwidth]{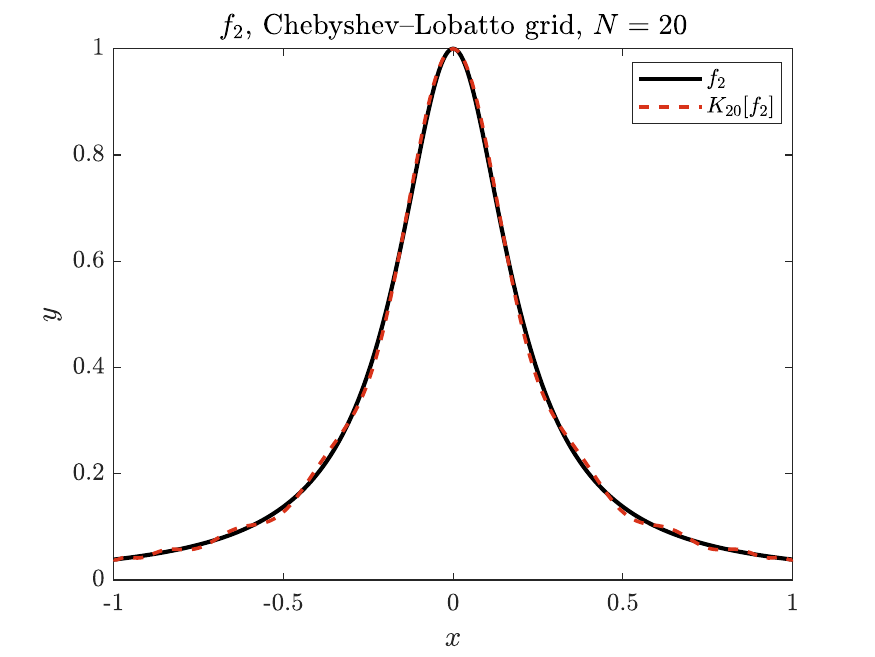}
\includegraphics[width=0.32\textwidth]{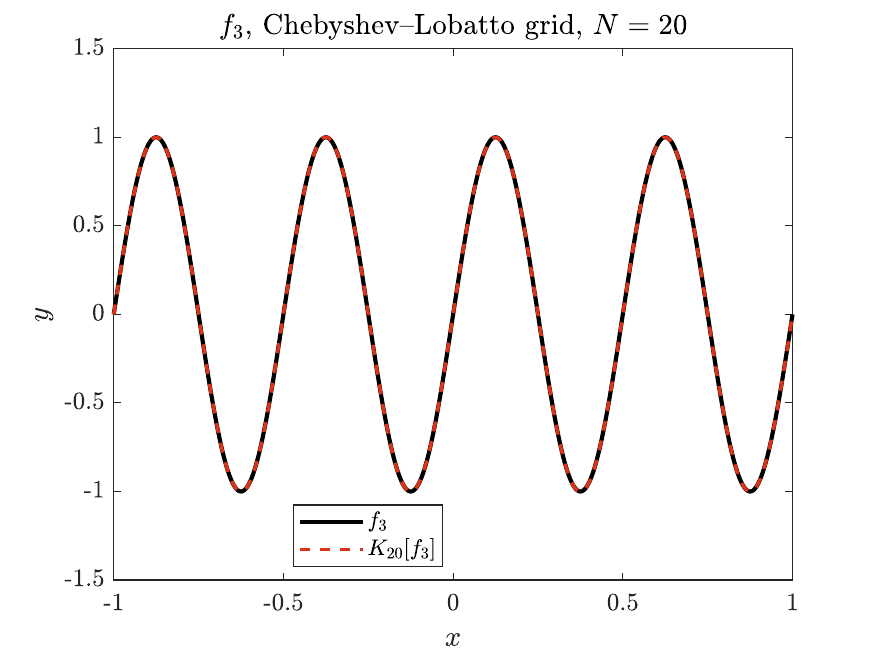}
\caption{Test functions $f_1$, $f_2$ and $f_3$, together with the corresponding
approximations obtained with $N=10$ (top) and $N=20$ (bottom) on the Chebyshev--Lobatto grid, using the
elementary family of intervals $\mathcal E$.}
\label{fig:new}
\end{figure}

Finally, we study the dependence of the method on the Jacobi parameters
$\alpha$ and $\beta$. To isolate the effect of the weight, we fix the
Chebyshev--Lobatto grid with $N=30$ and the elementary family of
intervals~\eqref{Ep}. Then we vary the pair $(\alpha,\beta)$ in the definition
of the Jacobi weight. The corresponding errors are reported in
Table~\ref{tab:jacobi-sensitivity}. The reference choice
$(\alpha,\beta)=(1/2,1/2)$, used in the previous tests, is included in the
comparison. The results show that the errors are only slightly affected by the choice of
the Jacobi parameters. In all cases, the errors remain essentially of the same
order of magnitude when the pair $(\alpha,\beta)$ is varied. The choice of the
Jacobi weight therefore has only a limited effect on the accuracy of the
method. The largest variation is observed for the function $f_1$ when the
difference between $\alpha$ and $\beta$ is large. In this case, passing from
$(\alpha,\beta)=(0,0)$ to $(\alpha,\beta)=(1,16)$ leads to the loss of about
one digit of accuracy. Nevertheless, the overall behaviour of the method
remains essentially unchanged.

\begin{table}[ht]
\centering
\begin{tabular}{cccc}
\toprule
$(\alpha,\beta)$ & $f_1$ & $f_2$ & $f_3$ \\
\midrule
$(0,0)$ & 2.220e-15 & 2.430e-03 & 7.967e-10 \\
$(1/2,1/2)$ & 2.665e-15 & 2.414e-03 & 7.729e-10 \\
$(2,1)$ & 2.220e-15 & 2.474e-03 & 7.298e-10 \\
$(1,2)$ & 2.665e-15 & 2.474e-03 & 7.298e-10 \\
$(4,1)$ & 2.220e-15 & 2.613e-03 & 6.939e-10 \\
$(1,4)$ & 1.776e-15 & 2.613e-03 & 6.939e-10 \\
$(8,2)$ & 2.665e-15 & 2.740e-03 & 6.208e-10 \\
$(2,8)$ & 1.132e-14 & 2.740e-03 & 6.208e-10 \\
$(16,1)$ & 9.326e-15 & 3.067e-03 & 2.281e-09 \\
$(1,16)$ & 1.177e-14 & 3.067e-03 & 2.443e-09 \\
\bottomrule
\end{tabular}
\caption{Sensitivity of the method with respect to the Jacobi parameters
$\alpha$ and $\beta$. The errors are computed with $N=30$ on the
Chebyshev--Lobatto grid, using the elementary family of intervals
$\mathcal E$.}
\label{tab:jacobi-sensitivity}
\end{table}

\section{Conclusions and Future Work}\label{sec:conclusions}

In this work, we have introduced a weighted derivative histopolation framework
on families of intervals. The degrees of freedom are defined by one scalar
normalization and by weighted integral moments of the derivative over
prescribed subintervals. We have proved that the resulting interpolation
problem is unisolvent on $\Pi_N$ when the interval family is
$\omega$-admissible and the normalization is nonzero on constants.
Equivalently, the derivative moments determine the polynomial up to an
additive constant, and the scalar normalization fixes this remaining degree
of freedom. This gives a simple and sharp criterion for the well-posedness of
the method, together with a complete characterization of the admissible scalar
normalizations. We have also studied interval families generated from a fixed grid. When the
endpoints of the intervals belong to the grid, admissibility can be checked
through the nonsingularity of the associated interval matrix. This finite
dimensional criterion depends only on the representation of the intervals in
terms of consecutive cells, and allows several admissible families of
intervals to be constructed in a purely algebraic way. For Jacobi weights, we have analyzed the corresponding weighted derivative
moment problem in Jacobi polynomial bases. The data matrix has a natural
block form which separates the constant polynomial from the nonconstant
polynomial part. The reduced derivative matrix is therefore the relevant
matrix for the nonconstant component. We have expressed this matrix in terms
of shifted Jacobi moment matrices, obtaining an algebraic description of the
method which is adapted to the weighted setting. In Chebyshev configurations, this structure becomes more explicit. For the
four classical Chebyshev families, suitable polynomial bases lead to diagonal
Gram matrices for the reduced derivative matrices. We have also identified
configurations in which the reduced derivative matrices admit exact sine
transform factorizations. These factorizations yield explicit singular values
and spectral condition numbers, and therefore allow the conditioning of the
method to be computed directly. The numerical experiments illustrate the behaviour of the method for different
interval families, grids, and Jacobi parameters. They show the accuracy of the
proposed schemes for the admissible configurations considered in the paper.
They also show the different behaviour of equispaced and Chebyshev--Lobatto
grids, especially in the presence of functions with Runge-type features. The
tests with different Jacobi parameters indicate that, for the examples
considered here, the accuracy is only slightly affected by the choice of the
weight. 

Future work will be devoted to extensions of the present construction beyond
the one-dimensional interval setting. A natural direction is to replace
intervals by suitable families of subdomains and to study weighted moments of
derivatives over higher dimensional elements. The main problem will be to
identify geometric admissibility conditions leading to unisolvent derivative
histopolation schemes, and to understand when the associated matrices have a
structure which allows stability and conditioning to be analyzed explicitly,
or through spectral distribution tools for matrix sequences~\cite{Capizzano:2003:GLT}.

\section*{Declarations}

\textbf{Corresponding author}\\
Federico Nudo, email federico.nudo@unical.it\\

\noindent
\textbf{Conflict of Interest}\\
The authors declare that they have no conflict of interest. Prof. Allal Guessab has been
retired since 1 September 2025 and is currently an independent researcher.\\

\noindent
\textbf{Funding statement}\\
This research was supported by the GNCS-INdAM 2026 project
\emph{``Metodi polinomiali e kernel per l'approssimazione da dati discreti e integrali con software OS''}
(CUP E53C25002010001).
The work of F. Nudo was funded by the European Union -- NextGenerationEU under the Italian National Recovery and Resilience Plan (PNRR), Mission 4, Component 2, Investment 1.2
\lq\lq Finanziamento di progetti presentati da giovani ricercatori\rq\rq,
pursuant to MUR Decree No.~47/2025.\\

\noindent
\textbf{Author Contributions}\\
Allal Guessab and Federico Nudo contributed equally to the conception, development, and writing of this manuscript.
All authors have read and approved the final version of the paper. For this reason, the order of authorship is alphabetical.\\

\noindent
\textbf{Acknowledgement}\\
This research was carried out as part of RITA \textquotedblleft Research ITalian network on Approximation'' and as part of the UMI group \enquote{Teoria dell'Approssimazione e Applicazioni}.\\

\noindent
\textbf{Data Availability}\\
No data were used in this study.

\bibliographystyle{spmpsci}
\bibliography{bibliography}

\end{document}